\documentclass[a4paper,11pt]{article}
\usepackage[latin1]{inputenc}
\usepackage[english]{babel}
\usepackage{amsmath}
\usepackage{amsfonts}
\usepackage{amssymb}
\usepackage{epsfig}
\usepackage{amsopn}
\usepackage{amsthm}
\usepackage{color}
\usepackage{graphicx}
\usepackage{enumerate}
\usepackage{mathrsfs}
\usepackage{cite}
\newtheorem{theorem}{Theorem}[section]
\newtheorem{corollary}[theorem]{Corollary}
\newtheorem{lemma}[theorem]{Lemma}
\newtheorem{proposition}[theorem]{Proposition}
\newtheorem{definition}[theorem]{Definition}
\newtheorem{remark}[theorem]{Remark}
\newtheorem*{theorem*}{Theorem}
\newtheorem*{lemma*}{Lemma}
\newtheorem*{remark*}{Remark}
\newtheorem*{definition*}{Definition}
\newtheorem*{proposition*}{Proposition}
\newtheorem*{corollary*}{Corollary}
\numberwithin{equation}{section}
\newcommand{\real}{\mathbb{R}}

\let\ced=\c         % cedilla
\def\x{\xi}

\def\qed{\,\unskip\kern 6pt \penalty 500
\raise -2pt\hbox{\vrule \vbox to8pt{\hrule width 6pt
\vfill\hrule}\vrule}\par}
\definecolor{darkblue}{rgb}{0.05, .05, .65}
\definecolor{darkgreen}{rgb}{0.1, .65, .1}
\definecolor{darkred}{rgb}{0.8,0,0}
\newcommand{\beqn}{\begin{equation}}
\newcommand{\eeqn}{\end{equation}}
\newcommand{\bear}{\begin{eqnarray}}
\newcommand{\eear}{\end{eqnarray}}
\newcommand{\bean}{\begin{eqnarray*}}
\newcommand{\eean}{\end{eqnarray*}}
\begin{document}
%%%%%%%%%%%%%%%%%%%%
%%%%%%%%%%%%%%%%%%%%

%%%%%%%%%%%%%%%%%%%%
%%%%%%%%%%%%%%%%%%%%
\title{\huge \bf Eternal solutions to a porous medium equation with strong nonhomogeneous absorption. Part II: Dead-core profiles}

\author{
\Large Razvan Gabriel Iagar\,\footnote{Departamento de Matem\'{a}tica
Aplicada, Ciencia e Ingenieria de los Materiales y Tecnologia
Electr\'onica, Universidad Rey Juan Carlos, M\'{o}stoles,
28933, Madrid, Spain, \textit{e-mail:} razvan.iagar@urjc.es}
\\[4pt] \Large Philippe Lauren\ced{c}ot\,\footnote{Laboratoire de Math\'ematiques (LAMA) UMR~5127, Universit\'e Savoie Mont Blanc, CNRS, F--73000 Chamb\'ery, France. \textit{e-mail:} philippe.laurencot@univ-smb.fr}\\ [4pt] \Large Ariel S\'{a}nchez\footnote{Departamento de Matem\'{a}tica
Aplicada, Ciencia e Ingenieria de los Materiales y Tecnologia
Electr\'onica, Universidad Rey Juan Carlos, M\'{o}stoles,
28933, Madrid, Spain, \textit{e-mail:} ariel.sanchez@urjc.es} \\[4pt]}
\date{\today}
\maketitle
%%%%%%%%%%%%%%%%%%%%
%%%%%%%%%%%%%%%%%%%%

%%%%%%%%%%%%%%%%%%%%
\begin{abstract}
Existence of a specific family of \emph{eternal solutions} in exponential self-similar form is proved for the following porous medium equation with strong absorption
$$
\partial_t u-\Delta u^m+|x|^{\sigma}u^q = 0 \;\;\text{ in }\;\; (0,\infty)\times\mathbb{R}^N,
$$
with $m>1$, $q\in(0,1)$ and $\sigma=2(1-q)/(m-1)$. Looking for solutions of the form
$$
u(t,x)=e^{-\alpha t}f(|x|e^{\beta t}), \qquad \alpha=\frac{2}{m-1}\beta,
$$
it is shown that, for $m+q>2$, there exists a unique exponent $\beta_*\in(0,\infty)$ for which there exists a one-parameter family of compactly supported profiles presenting a \emph{dead core}. The precise behavior of the solutions at their interface is also determined. Moreover, these solutions show the optimal limitations for the finite time extinction property of genuine non-negative solutions to the Cauchy problem, studied in previous works.
\end{abstract}
%%%%%%%%%%%%%%%%%%%%

\smallskip


\smallskip


%%%%%%%%%%%%%%%%%%%%
%%%%%%%%%%%%%%%%%%%%
\section{Introduction and main results}
%%%%%%%%%%%%%%%%%%%%
%%%%%%%%%%%%%%%%%%%%

In the present work we complete the quest, started in the previous paper \cite{IL24}, of constructing and classifying some specific solutions in exponential form to the following quasilinear diffusion equation with strong absorption
\begin{equation}\label{eq1}
\partial_t u-\Delta u^m+|x|^{\sigma}u^q=0, \qquad (t,x)\in(0,\infty)\times\real^N,
\end{equation}
in the range of exponents
\begin{equation}\label{range.exp}
m>1, \qquad q\in(0,1), \qquad \sigma=\sigma_*:=\frac{2(1-q)}{m-1}.
\end{equation}
Assuming furthermore that $m+q>2$, this paper is devoted to the construction of a family of \emph{eternal solutions} presenting dead cores and interfaces, the former being a rather unexpected feature. The existence of these solutions, together with the more usual self-similar solutions with a decreasing profile established in \cite{IL24} for the same critical exponent $\sigma_*$ as in \eqref{range.exp}, proves that the dynamics of Eq.~\eqref{eq1} is expected to be very complex and interesting, since according to the general experience on absorption-diffusion equations, self-similar solutions are the expected profiles for the large time behavior of general solutions.

The most interesting aspect related to Eq.~\eqref{eq1} is the competition between its two terms: on the one hand, a quasilinear diffusion term featuring a mass conservation property and, on the other hand, an absorption term, depending also on the space variable, which introduces a dissipation of the $L^1$-norm of solutions. This competition gives rise to a number of critical exponents that have been identified, splitting the analysis of the equation into different ranges where the dynamics is led by either of the two effects, or by a kind of balance between them. Indeed, letting $\sigma=0$ in Eq.~\eqref{eq1}, three such ranges have been found for $q>1$, namely

$\bullet$ $q>m+2/N$, where the diffusion controls the dynamics of Eq.~\eqref{eq1}, which is rather easily described by an asymptotic simplification as $t\to\infty$, see for example \cite{KP86}.

$\bullet$ $m<q<m+2/N$, where a balance between diffusion and absorption produces some \emph{very singular solutions} which are specific to this range. This range is now also well understood, see for example \cite{KP86, PT86, KU87, KV88, KPV89, Le97, Kwak98} and references therein.

$\bullet$ $1<q<m$, where the absorption term becomes to dominate over the diffusion for large times and leads to a localization of the compactly supported solutions (contrasting with the expansion of the support induced by the diffusion), established at least in one space dimension in \cite{PT85, CV96, CV99}. There are still a number of open problems to be understood in this range.

However, probably the most striking and mathematically interesting range is the one known as \emph{strong absorption}; that is, $m>1$ and $q\in(0,1)$. As the name suggests, in this range, the absorption term becomes very strong and governs the evolution, leading to a number of mathematical phenomena that are not present in the other ranges described above. Perhaps the most striking feature specific to this range is the \emph{instantaneous shrinking of the supports} of solutions: even starting with an initial condition $u_0$ such that $u_0(x)>0$ for any $x\in\real^N$, under suitable conditions on $u_0$, solutions to Eq.~\eqref{eq1} become instantaneously compactly supported; that is, ${\rm supp}\,u(t)\subseteq B(0,R(t))$ for some $R(t)>0$ and for any $t>0$. Letting $\sigma=0$ in Eq.~\eqref{eq1}, such an unexpected property has been established in \cite{EK79, Ka84, Abd98}. Besides, another typical feature of solutions in this range is the \emph{finite time extinction}: we say that a solution to Eq.~\eqref{eq1} vanishes in finite time if there is $T_e\in(0,\infty)$ such that $u(t)\not\equiv0$ for $t\in(0,T_e)$, but $u(T_e)\equiv0$. Letting still $\sigma=0$ in Eq.~\eqref{eq1}, this vanishing property stems from the ordinary differential equation $\partial_t u=-u^q$ obtained by neglecting the diffusion term, emphasizing thus the domination of the absorption, as proved in \cite{Ka75, Ka84}. Due to the difficulty of the problem, a description of the large time behavior of general solutions is only available when $m+q=2$, see \cite{GV94}.

Returning to Eq.~\eqref{eq1} with $m>1$, $q\in(0,1)$ and $\sigma>0$, it was proved in \cite{Belaud01, BeSh07} that any solution to Eq.~\eqref{eq1} posed in a bounded domain $\Omega\subset\real^N$ with homogeneous Neumann boundary condition vanishes in finite time provided $0<\sigma<\sigma_*$. The analogous result with homogeneous Dirichlet boundary conditions is established in \cite{BeSh22}, following similar results for the semilinear equation ($m=1$) in \cite{BHV01, BD10}.

These results determined the authors to explore the range $\sigma>\sigma_*$ in the recent work \cite{ILS24}. On the one hand, it has been shown therein that instantaneous shrinking of the supports is in force for any $\sigma>0$ and for any initial condition $u_0\in L^{\infty}$, a fact that emphasizes the strength of the weight $|x|^{\sigma}$, since for $\sigma=0$ the condition $u_0(x)\to0$ as $|x|\to\infty$ is required for the shrinking. On the other hand, if $\sigma>\sigma_*$, it is shown in\cite{ILS24} that a very wide class of initial conditions, satisfying the property $u_0(x)>0$ in a small neighborhood of the origin, give rise to solutions that \emph{do not vanish in finite time} despite the shrinking of their supports and, in fact, stabilize to a self-similar solution as $t\to\infty$, in stark contrast with the range $\sigma<\sigma_*$. Moreover, the short note \cite{IL23} proves that any bounded solution vanishes in finite time if $\sigma<\sigma_*$ and gives rather sharp conditions on the initial data $u_0\in L^{\infty}(\real^N)$ for finite time extinction to take place when $\sigma>\sigma_*$.

The previous comments motivate a deeper study of the critical exponent $\sigma=\sigma_*$, seen as the limiting exponent between a range where finite time extinction always occurs and another one where this property strongly depends on the initial condition. Thus, the current paper completes the study started in \cite{IL24} providing self-similar solutions with an exponential time decay, but not finite time extinction, which also play the role of optimal conditions for positivity at any $t>0$, as it is explained below.

\medskip

\noindent \textbf{Main results}. We are looking for some special solutions to~\eqref{eq1} with $m$, $q$ and $\sigma$ as in \eqref{range.exp} having an \emph{exponential self-similar form}; that is,
\begin{equation}\label{SSS}
u(t,x)=e^{-\alpha t}f(|x|e^{\beta t}), \qquad (t,x)\in \real\times\real^N.
\end{equation}
Such solutions are also called \emph{eternal}, as they are defined for all $t\in\real$. Introducing the ansatz~\eqref{SSS} into Eq.~\eqref{eq1}, letting $\xi=|x|e^{\beta t}$ and performing some direct calculations, we readily find that the self-similar exponents must satisfy the condition
\begin{equation}\label{SS.exp}
\alpha=\frac{2}{m-1}\beta,
\end{equation}
where $\beta$ becomes a free parameter for our analysis, while the profile $f$ solves the differential equation
\begin{equation}\label{ODE}
(f^m)''(\xi)+\frac{N-1}{\xi}(f^m)'(\xi)+\alpha f(\xi)-\beta\xi f'(\xi)-\xi^{\sigma}f^q(\xi)=0, \qquad \xi>0.
\end{equation}
Let us notice at this point that, given $\lambda>0$ and a solution $f$ to~\eqref{ODE}, we obtain by direct calculations that the rescaled function
\begin{equation}\label{resc}
f_{\lambda}(\xi):=\lambda f(\lambda^{-(m-1)/2}\xi), \qquad \xi\ge 0,
\end{equation}
is also a solution to \eqref{ODE}.

In a first part of this research, published in \cite{IL24}, we proved that there is a unique pair of positive exponents $(\alpha^*,\beta^*)$ satisfying \eqref{SS.exp} such that there exists a unique self-similar profile $f^*$ with the properties $f^*(0)=1$, $(f^*)'(0)=0$, and there is $\xi^*\in(0,\infty)$ such that $f(\xi)>0$ for $\xi\in[0,\xi^*)$, $f(\xi^*)=0$, $(f^m)'(\xi^*)=0$ and $f$ is decreasing on its positivity interval $[0,\xi^*]$. The self-similar solution \eqref{SSS} with profile $f^*$ presents thus analogous properties to the compactly supported self-similar solution with algebraic time decay constructed in \cite{ILS24}, which proved to be fundamental for the large time behavior of general solutions to Eq. \eqref{eq1} when $\sigma>\sigma_*$.

A second and rather surprising type of eternal self-similar solution comes from the availability of profiles presenting a \emph{dead core}. More precisely, we are looking for profiles $f$ solving~\eqref{ODE} with compact support $[\xi_*,\xi_0]\subset (0,\infty)$ which satisfy the standard contact condition $(f^m)'(\xi_*)=(f^m)'(\xi_0)=0$ at the edges of their support, see the beginning of Section~\ref{sec.dci} for precise definitions. Let us again notice that, if $f$ is such a profile, then the rescaled functions $f_{\lambda}$ defined in \eqref{resc} for any $\lambda>0$ fulfill the same properties of dead core and interface, in this case with
$$
f_{\lambda}(\xi)>0, \quad 0<\lambda^{(m-1)/2}\xi_*<\xi<\lambda^{(m-1)/2}\xi_0<\infty.
$$
Thus we have to fix the initial point of the support of the profile as, for example, $\xi_*=1$ in order to be able to show uniqueness. We can now state our existence and uniqueness result for such self-similar solutions, which holds true only for $m+q>2$.

%%%%%%%%%%%%%%%%%%%%
\begin{theorem}\label{th.1}
Let $m$, $q$ and $\sigma$ as in~\eqref{range.exp} be such that $m+q>2$. Then there exist a unique exponent $\beta_*$ (and corresponding $\alpha_*$ given by \eqref{SS.exp}) and a unique non-negative radially symmetric self-similar solution to~\eqref{eq1} in exponential form
$$
U_*(t,x)=e^{-\alpha_*t}f_*(|x|e^{\beta_*t}), \qquad (t,x)\in \real\times\real^N,
$$
where the profile $f_*$ is a solution to~\eqref{ODE} (with $\beta=\beta_*$) such that ${\rm supp}\,f_*=[1,\xi_0]$ for some $\xi_0\in(1,\infty)$.
\end{theorem}
%%%%%%%%%%%%%%%%%%%%

The availability of such dead core profiles is a interesting feature of Eq.~\eqref{eq1}. Indeed, formation of dead cores has been investigated in diffusion equations involving a reaction term, starting from the paper \cite{BS84}. More recently, dead cores have been also discovered as properties of solutions to diffusion equations with strong absorption, see for example \cite{GLS10, S11}. However, these works deal with problems featuring a finite time extinction of the solutions, and it is shown therein that the rate of formation of dead cores is \emph{non-self-similar}. The solution we construct in Theorem~\ref{th.1} differs thus from these previous works since it is a solution to the Cauchy problem in $\real^N$ and does not vanish in finite time, the dead core being present for any $t\in\real$ with edges evolving at an exponential rate. Self-similar solutions presenting dead cores (and a finite time blow-up) are also uncovered for quasilinear reaction-diffusion equations in recent works by two of the authors \cite{IS22, IS24}.

Since we have identified self-similar solutions with positive values at $x=0$ in \cite{IL24} and Theorem~\ref{th.1} provides a family of solutions (modulo the rescaling~\eqref{resc}) presenting a dead core in a neighborhood of the origin, a natural question arises with respect to the availability of profiles (and thus self-similar solutions in exponential form) such that $f(0)=0$ but $f(\xi)>0$ in a right neighborhood of the origin. This is the subject of the next result, where, in order to ensure uniqueness and in view of~\eqref{resc}, we have now to fix the interface to, for example, $\xi_0=1$.

%%%%%%%%%%%%%%%%%%%%
\begin{theorem}\label{th.2}
Let $m$, $q$ and $\sigma$ as in~\eqref{range.exp} be such that $m+q>2$. Then there exists an exponent $\beta_0>0$ such that, for any $\beta\in[\beta_0,\infty)$ and corresponding $\alpha\in[\alpha_0,\infty)$ given by \eqref{SS.exp}, there exists a unique non-negative radially symmetric self-similar solution in exponential form
$$
U(t,x)=e^{-\alpha t}f(|x|e^{\beta t}), \qquad (t,x)\in\mathbb{R}\times\mathbb{R}^N,
$$
where the profile $f$ is a solution to~\eqref{ODE} such that ${\rm supp}\,f_*=[0,1]$. Moreover, the profile $f$ has the following local behavior at the origin:
\begin{equation}\label{beh.P2}
f(\xi)=\left[\frac{(m-1)^2}{2m[N(m-1)+2]}\right]^{1/(m-q)}\xi^{2/(m-1)}+o(\xi^{2/(m-1)}) \quad {\rm as} \ \xi\to0.
\end{equation}
\end{theorem}
%%%%%%%%%%%%%%%%%%%%

%%%%%%%%%%%%%%%%%%%%
\begin{remark}\label{rem.1}
	As it will follow from the proofs, the self-similar exponents $\beta^*$, $\beta_*$, and $\beta_0$ of the solutions obtained in \cite{IL24}, Theorem~\ref{th.1}, and Theorem~\ref{th.2}, respectively, satisfy $\beta^*<\beta_*<\beta_0$.
\end{remark}
%%%%%%%%%%%%%%%%%%%%

%%%%%%%%%%%%%%%%%%%%
\begin{remark}\label{rem.1bis}
The solutions given by Theorem~\ref{th.2} behave as 
$$
U(t,x)\sim\left[\frac{(m-1)^2}{2m[N(m-1)+2]}\right]^{1/(m-q)}|x|^{2/(m-1)}, \quad {\rm as} \ t\to-\infty,
$$
while solutions given by Theorem~\ref{th.1} satisfies $U_*(t,x) = 0$ as $t\to-\infty$, for any $x\in\real^N$.
\end{remark}
%%%%%%%%%%%%%%%%%%%%

Let us point out that the solutions constructed in Theorems~\ref{th.1} and~\ref{th.2} are \emph{borderline cases} emphasizing the optimality of the conditions established in \cite{IL23} for finite time extinction to take place. Indeed, on the one hand, for non-negative initial conditions $u_0$ supported in an annulus $\{x\in\real^N\ :\ r < |x|<R\}$ with $0<r<R$, it follows from \cite[Corollary~3.2]{IL23} that the corresponding solution to~\eqref{eq1} (with parameters as in~\eqref{range.exp}) vanishes in finite time if the amplitude $\|u_0\|_{\infty}$ is sufficiently small, and Theorem~\ref{th.1} implies that such a smallness criterion cannot be removed. On the other hand, when $u_0$ is positive in $\real^N\setminus\{0\}$ with $u_0(0)=0$, a condition guaranteeing finite time extinction of the corresponding solution to~\eqref{eq1} is that there are $a>\sigma_*/(1-q)=2/(m-1)$ and $C>0$ such that $u_0(x)\le C |x|^a$ for $x\in\real^N$, along with the smallness of $\|u_0\|_\infty$, see \cite[Corollary~3.2]{IL23}. According to~\eqref{beh.P2}, the solutions constructed in Theorem~\ref{th.2} show the sharpness of this criterion for extinction in finite time to occur.

\medskip

\textbf{Comments on the techniques.} While the proofs in the first part of this research \cite{IL24} rely on a classical shooting method, such an approach does not seem applicable here. Thus, we employ a different and more involved technique, starting from an alternative formulation of the differential equation~\eqref{ODE} as an autonomous dynamical system in $\mathbb{R}^2$ obtained through a transformation, as indicated in Section~\ref{subsec.aaf}. Some preliminary results are first needed in order to exploit fully this transformation. We begin with the identification of the behavior of the solutions to~\eqref{ODE} at the edges of their support when the latter is included in $[0,\infty)$, which is performed in Section \ref{sec.dci}. With the aid of these preliminary results, we prove in Section~\ref{sec.dynamic} that these solutions correspond to complete orbits of the autonomous dynamical system, connecting two critical points of it.

The remaining sections are devoted to the analysis of the trajectories of this dynamical system. Let us mention here that the fundamental tool in this analysis is the identification of \emph{invariant regions}, with boundaries playing the role of barriers for the trajectories of the system, limiting their $\omega$-limit sets. On the one hand, we show in Section~\ref{subsec.irm} that the availability of such barriers then implies a monotone variation (with an opposite monotonicity) of the trajectories connecting from/to the two critical points of the dynamical system unfolding the dead core behavior, respectively the interface behavior, with respect to the parameter $\beta$. On the other hand, a rather technical analysis, performed in Section~\ref{subsec.cd}, is needed in order to establish the \emph{continuous dependence} when changing the parameter $\beta$, as we in fact deal with a one-parameter family of distinct dynamical systems. The proof of Theorem \ref{th.1} then follows from the opposite monotonicity, the continuous dependence and an application of Bolzano's Theorem. As for Theorem~\ref{th.2}, its proof is a consequence of a part of the analysis establishing the configuration of the trajectories for $\beta$ large, given in Section~\ref{subsec.cppsK}.

%%%%%%%%%%%%%%%%%%%%
%%%%%%%%%%%%%%%%%%%%
\section{Dead cores and interfaces}\label{sec.dci}
%%%%%%%%%%%%%%%%%%%%
%%%%%%%%%%%%%%%%%%%%

In this section, we rigorously define what we mean by dead core and interface and we give some preliminary results concerning the local behavior of a profile near these two edges of its support. These results will be then employed at the end of the section, when an alternative formulation of the differential equation \eqref{ODE} is introduced, in order to identify the critical points of interest in the forthcoming study. We thus start by defining the two main features of the solutions considered in this paper.

%%%%%%%%%%%%%%%%%%%%
\begin{definition}\label{def.dcp}
Let $\beta>0$.
	
\medskip
	
(i) A solution $f$ to~\eqref{ODE} has a \emph{dead core} if there exists $\xi_*\ge 0$ and $\delta_*>0$ such that
%\begin{subequations}\label{dcp}
	\begin{equation*}
		f(\xi)=0 \;\text{ for any }\; \xi\in[0,\xi_*], \qquad f(\xi)>0 \;\text{ for any }\; \xi\in(\xi_*,\xi_*+\delta_*], %\label{dcpa}
	\end{equation*}
	and
	\begin{equation*}
		(f^m)'(\xi_*)=0. %\label{dcpb}
	\end{equation*}
%\end{subequations}

\medskip

(ii) A solution $f$ to~\eqref{ODE} has an \emph{interface} if there exists $\xi_0>0$ and $\delta_0\in (0,\xi_0)$ such that
%\begin{subequations}\label{itf}
	\begin{equation*}
		f(\xi)=0 \;\text{ for any }\; \xi\in[\xi_0,\xi_0+\delta_0], \qquad f(\xi)>0 \;\text{ for any }\; \xi\in[\xi_0-\delta_0,\xi_0), %\label{itfa}
	\end{equation*}
	and
	\begin{equation*}
		(f^m)'(\xi_0)=0. %\label{itfb}
	\end{equation*}
%\end{subequations}
\end{definition}
%%%%%%%%%%%%%%%%%%%%

Let us mention here that, with an ``abuse of language", we also call a dead core the vanishing of the profile only at $\xi_*=0$, as included in part (i) of Definition \ref{def.dcp}. We analyze in the next subsections the local behavior of a profile $f$ as $\xi\to\xi_*$ when $\xi_*>0$ and as $\xi\to\xi_0$, recalling that $m+q>2$.

%%%%%%%%%%%%%%%%%%%%
%%%%%%%%%%%%%%%%%%%%
\subsection{Behavior at a dead core}\label{subsec.bdc}
%%%%%%%%%%%%%%%%%%%%
%%%%%%%%%%%%%%%%%%%%

The next result makes precise the local behavior of a profile with a dead core near its left edge of the support.

%%%%%%%%%%%%%%%%%%%%
\begin{lemma}\label{lem.p001}
	Let $\beta>0$ and consider a solution $f$ to \eqref{ODE} having a dead core at $\xi_*>0$ in the sense of Definition~\ref{def.dcp}. Then $f$ is increasing in a right neighborhood of $\xi_*$ with
	\begin{subequations}\label{p001}
		\begin{equation}
			\lim_{\xi\searrow\xi_*} \big(f^{m-1} \big)'(\xi) = \frac{m-1}{m} \beta\xi_*, \label{p001a}
		\end{equation}
		and, as $\xi\searrow\xi_*$,
		\begin{equation}
			f(\xi) = \left( \frac{m-1}{m} \beta \xi_* \right)^{1/(m-1)} (\xi-\xi_*)^{1/(m-1)} + o\big( (\xi-\xi_*)^{1/(m-1)}  \big). \label{p001b}
		\end{equation}
	\end{subequations}
\end{lemma}
%%%%%%%%%%%%%%%%%%%%

\begin{proof}
	Set $F=f^m$ and
	\begin{equation}
		H(\xi) := \xi^{N-1} F'(\xi) - \beta \xi^N f(\xi), \qquad \xi \in [\xi_*,\xi_*+\delta_*]. \label{p002}
	\end{equation}
	On the one hand, Definition~\ref{def.dcp} implies that $H(\xi_*)=0$. On the other hand, it follows from~\eqref{ODE} that
	\begin{equation}
		H'(\xi) = -(\alpha+N\beta) \xi^{N-1} f(\xi) + \xi^{\sigma+N-1} f^q(\xi), \qquad \xi\in [\xi_*,\xi_*+\delta_*]. \label{p003}
	\end{equation}
	Equivalently,
	\begin{equation*}
		H'(\xi) = \xi^{N-1} f^q(\xi) \left[ \xi^\sigma  -(\alpha+N\beta) f^{1-q}(\xi) \right], \qquad \xi\in [\xi_*,\xi_*+\delta_*],
	\end{equation*}
	and, since
	\begin{equation*}
		\lim_{\xi\searrow\xi_*} \left[ \xi^\sigma  -(\alpha+N\beta) f^{1-q}(\xi) \right] = \xi_*^\sigma>0
	\end{equation*}
	and $\xi^{N-1} f^q(\xi) >0$ for $\xi\in (\xi_*,\xi_*+\delta_*]$ by Definition~\ref{def.dcp}, we conclude that there is $\delta\in (0,\delta_*)$ such that $H'>0$ on $(\xi_*,\xi_*+\delta)$. Consequently, $H(\xi)>H(\xi_*)=0$ for $\xi\in (\xi_*,\xi_*+\delta)$ and we deduce from the positivity of $f$ on $(\xi_*,\xi_*+\delta)$ and the identity
	\begin{equation*}
		\xi^{N-1} F'(\xi) = H(\xi) + \beta \xi^N f(\xi), \qquad \xi\in (\xi_*,\xi_*+\delta),
	\end{equation*}
	that
	\begin{equation}
		F'(\xi)>0, \qquad \xi\in (\xi_*,\xi_*+\delta). \label{p004}
	\end{equation}
	Now, dropping the positive term in the right-hand side of~\eqref{p003} and integrating on $(\xi_*,\xi)$ for  $\xi\in (\xi_*,\xi_*+\delta)$, we infer from the monotonicity~\eqref{p004} of $f$ and the property $H(\xi_*)=0$ that
	\begin{align*}
		H(\xi) & = H(\xi) -H(\xi_*) \ge - (\alpha+\beta N) \int_{\xi_*}^\xi s^{N-1} f(s)\ ds \\
		& \ge -(\alpha+N\beta) f(\xi) \int_{\xi_*}^\xi s^{N-1}\ ds \ge - (\alpha+N\beta) f(\xi) \xi^{N-1} (\xi-\xi_*).
	\end{align*}
	Therefore, recalling the definition~\eqref{p002} of $H$ and dividing by $\xi^{N-1} f(\xi)>0$, we obtain
	\begin{equation*}
		\frac{m}{m-1} \big( f^{m-1} \big)'(\xi) \ge \beta\xi - (\alpha+N\beta) (\xi-\xi_*), \qquad \xi\in (\xi_*,\xi_*+\delta),
	\end{equation*}
	from which we readily deduce that
	\begin{equation}
		\liminf_{\xi\searrow\xi_*} \big( f^{m-1} \big)'(\xi) \ge \frac{m-1}{m} \beta\xi_*. \label{p005}
	\end{equation}
	According to~\eqref{p005}, we may assume that, after possibly reducing the value of $\delta>0$,
	\begin{equation*}
		\big( f^{m-1} \big)'(\xi) \ge \frac{m-1}{2m} \beta\xi_*, \qquad \xi\in (\xi_*,\xi_*+\delta).
	\end{equation*}
	Integration of the above inequality on $(\xi_*,\xi)$ for $\xi\in (\xi_*,\xi_*+\delta)$ gives
	\begin{equation}
		f^{m-1}(\xi) \ge \frac{m-1}{2m} \beta\xi_* (\xi-\xi_*), \qquad \xi\in (\xi_*,\xi_*+\delta). \label{p006}
	\end{equation}
	We next infer from~\eqref{p003} and~\eqref{p004} that, for $\xi\in (\xi_*,\xi_*+\delta)$,
	\begin{equation*}
		H(\xi) \le \int_{\xi_*}^\xi s^{\sigma+N-1} f^q(s)\ ds \le \xi^{\sigma+N-1} f^q(\xi) (\xi-\xi_*).
	\end{equation*}
	Hence, after dividing by $\xi^{N-1} f(\xi)>0$ and using~\eqref{p002} and~\eqref{p006},
	\begin{align*}
		\frac{m}{m-1} \big( f^{m-1} \big)'(\xi) & \le \beta \xi + \xi^\sigma f^{q-1}(\xi) (\xi-\xi_*)\\
		& \le \beta\xi + \left[ \frac{m-1}{2m} \beta\xi_* \right]^{(q-1)/(m-1)} \xi^\sigma (\xi-\xi_*)^{(m+q-2)/(m-1)}.
	\end{align*}
	Recalling that $m+q>2$, we may let $\xi\searrow\xi_*$ in the above inequality and conclude that
	\begin{equation*}
		\limsup_{\xi\searrow\xi_*} \big( f^{m-1} \big)'(\xi) \le \frac{m-1}{m} \beta\xi_*.
	\end{equation*}
	Gathering~\eqref{p005} and the above inequality completes the proof of~\eqref{p001a}. Integrating the latter and using $f^{m-1}(\xi_*)=0$ give~\eqref{p001b}.
\end{proof}

%%%%%%%%%%%%%%%%%%%%
\begin{remark}\label{rem.2}
	As one can readily notice, the previous proof cannot be extended to the limiting case $\xi_*=0$ corresponding to profiles described in Theorem~\ref{th.2}. The local behavior~\eqref{beh.P2} will be thus directly derived from the analysis of a dynamical system, as shown in the final section of the paper.
\end{remark}
%%%%%%%%%%%%%%%%%%%%

%%%%%%%%%%%%%%%%%%%%
%%%%%%%%%%%%%%%%%%%%
\subsection{Behavior at an interface}\label{subsec.bi}
%%%%%%%%%%%%%%%%%%%%
%%%%%%%%%%%%%%%%%%%%

The interface behavior for compactly supported solutions to \eqref{ODE} is given in the following statement.

%%%%%%%%%%%%%%%%%%%%
\begin{lemma}\label{lem.p002}
	Let $\beta>0$ and consider a solution $f$ to \eqref{ODE} having an interface at $\xi_0>0$ in the sense of Definition~\ref{def.dcp}. Then $f$ is decreasing in a left neighborhood of $\xi_0$ with
	\begin{subequations}\label{p007}
		\begin{equation}
			\lim_{\xi\nearrow\xi_0} \big(f^{m-1} \big)'(\xi) = 0, \label{p007a}
		\end{equation}
		and, as $\xi\nearrow\xi_0$,
		\begin{equation}
			f(\xi) = \left( \frac{1-q}{\beta} \xi_0^{\sigma-1} \right)^{1/(1-q)} (\xi_0-\xi)^{1/(1-q)} + o\big( (\xi_0-\xi)^{1/(1-q)}  \big). \label{p007b}
		\end{equation}
	\end{subequations}
\end{lemma}
%%%%%%%%%%%%%%%%%%%%

\begin{proof}
	Since
	\begin{equation*}
		\lim_{\xi\nearrow\xi_0} \left[ \xi^\sigma - \alpha f^{1-q}(\xi) \right] = \xi_0^\sigma>0,
	\end{equation*}
	we may assume that, after possibly reducing the value of $\delta_0>0$,
	\begin{equation}
		\xi^\sigma - \alpha f^{1-q}(\xi) > 0, \qquad \xi \in [\xi_0-\delta_0,\xi_0]. \label{p008}
	\end{equation}
	Assume for contradiction that there is $\xi_1\in (\xi_0-\delta_0,\xi_0)$ such that $\big( f^m \big)'(\xi_1)=f'(\xi_1)=0$. We then infer from~\eqref{ODE} and~\eqref{p008} that
	\begin{equation*}
		F''(\xi_1)  = - \alpha f(\xi_1) + \xi_1^\sigma f^q(\xi_1) = f^q(\xi_1) \left[ \xi_1^\sigma - \alpha f^{1-q}(\xi_1) \right]>0,
	\end{equation*}
	recalling that $f(\xi_1)>0$. Consequently, $f'$ is positive in a right neighborhood of $\xi_1$ and we set
	\begin{equation*}
		\xi_2 := \inf\left\{ \xi\in (\x_1,\xi_0)\ :\ f'(\xi)=0 \right\} > \xi_1.
	\end{equation*}	
	Since $f(\xi)\ge f(\xi_1)>0$ for $\xi\in [\xi_1,\xi_2]$, we necessarily have $\xi_2<\xi_0$ and we use again~\eqref{ODE}, \eqref{p008}, and the positivity of $f(\xi_2)$ to deduce that
	\begin{equation*}
		F''(\xi_2) = f^q(\xi_2) \left[ \xi_2^\sigma - \alpha f^{1-q}(\xi_2) \right]>0.
	\end{equation*}
	However, since $F'(\xi)\ge 0=F'(\xi_1)$ for $\xi\in (\xi_1,\xi_2)$, there holds $F''(\xi_2)\le 0$, and a contradiction. Consequently,
	\begin{equation}
		F'<0 \;\;\text{ and }\;\; f'<0 \;\;\text{ on }\;\; (\xi_0-\delta_0,\xi_0). \label{p009}
	\end{equation}
	
	Next, using once more the function $H$ defined by
	\begin{equation*}
		H(\xi) := \xi^{N-1} F'(\xi) - \beta \xi^N f(\xi), \qquad \xi \in [\xi_0-\delta_0,\xi_0], %\label{p010}
	\end{equation*}
	with derivative
	\begin{equation*}
		H'(\xi) = -(\alpha+N\beta) \xi^{N-1} f(\xi) + \xi^{\sigma+N-1} f^q(\xi), \qquad \xi\in [\xi_0-\delta_0,\xi_0], %\label{p011}
	\end{equation*}
	we infer from the non-negativity and monotonicity~\eqref{p009} of $f$ that, for $\xi\in (\xi_0-\delta_0,\xi_0)$,
	\begin{align}
		\beta \xi^N f(\xi) & \le -\xi^{N-1} F'(\xi) + \beta \xi^N f(\xi) = H(\xi_0) - H(\xi) \nonumber\\
		& \le G(\xi) := \int_\xi^{\xi_0} s^{\sigma+N-1} f^q(s)\ ds. \label{p012}
	\end{align}
	We next deduce from~\eqref{p012} that
	\begin{equation*}
		- \beta^q \frac{G'(\xi)}{\xi^{N(1-q)+\sigma-1}} = \big( \beta \xi^N f(\xi) \big)^q \le G^q(\xi) ;
	\end{equation*}
	that is,
	\begin{equation*}
		- \frac{1}{1-q} \big( G^{1-q} \big)'(\xi) \le \beta^{-q} \xi^{N(1-q)+\sigma-1}, \qquad \xi\in [\xi_0-\delta_0,\xi_0].
	\end{equation*}
	Integrating the above inequality over $(\xi,\xi_0)$ for $\xi\in(\xi_0-\delta_0,\xi_0)$, we obtain
	\begin{equation*}
		\frac{G^{1-q}(\xi)}{1-q} \le \frac{\xi_0^{N(1-q)+\sigma} - \xi^{N(1-q)+\sigma}}{[N(1-q)+\sigma] \beta^q},
	\end{equation*}
	whence, by~\eqref{p012},
	\begin{equation}
		f^{1-q}(\xi) \le \frac{(1-q)}{\beta} \frac{\xi_0^{N(1-q)+\sigma} - \xi^{N(1-q)+\sigma}}{N(1-q)+\sigma}\xi^{-N(1-q)}, \qquad \xi\in (\xi_0-\delta_0,\xi_0). \label{p013}
	\end{equation}
	In particular,
	\begin{equation}
		\limsup_{\xi\nearrow\xi_0} \frac{f^{1-q}(\xi)}{\xi_0-\xi} \le \frac{1-q}{\beta} \xi_0^{\sigma-1}. \label{p014}
	\end{equation}
	To complete the proof, we observe that, thanks to~\eqref{p013}, we may argue as in \cite[Lemma~4.9]{ILS24} to prove that
	\begin{equation}
		\limsup_{\xi\nearrow\xi_0} \big( f^{m-1} \big)'(\xi) = 0. \label{p015}
	\end{equation}
	We next combine~\eqref{p014} and~\eqref{p015} and proceed as in the proof of \cite[Proposition~4.11]{ILS24} to establish~\eqref{p007}.
\end{proof}

%%%%%%%%%%%%%%%%%%%%
%%%%%%%%%%%%%%%%%%%%
\subsection{An alternative formulation}\label{subsec.aaf}
%%%%%%%%%%%%%%%%%%%%
%%%%%%%%%%%%%%%%%%%%

In the next lines, we introduce a transformation which maps solutions to~\eqref{ODE} presenting dead core and interface onto trajectories of an autonomous dynamical system, allowing thus for the employment of techniques specific to the theory of dynamical systems. In order to simplify the notation, let us fix some constants that will come frequently into play in the subsequent analysis:
\begin{equation*}
\mu:= m+q-2>0, \qquad \nu:= \frac{m-1}{m+q-2} = 1 + \frac{1-q}{\mu}>1. %\label{mqg2}
\end{equation*}
Let $\beta>0$ and consider a solution $f$ to~\eqref{ODE} with a dead core at $\xi=\xi_*>0$, an interface at $\xi_0>\xi_*$, and satisfying $f>0$ on $(\xi_*,\xi_0)$. Borrowing ideas from \cite{IS22a, IS22b, J53}, we define the functions
\begin{equation}\label{change.large}
	\begin{split}
		\mathcal{Y}_f(\xi) & := \frac{2m}{\alpha}\xi^{-1} f^{m-2}(\xi) f'(\xi), \qquad \xi\in (\xi_*,\xi_0), \\
		\mathcal{Z}_f(\xi) & := \left[\frac{2m}{\alpha}\right]^{1/\nu} \xi^{-2/\nu} f^{\mu}(\xi), \qquad \xi\in (\xi_*,\xi_0),
	\end{split}
\end{equation}
with $\alpha=2\beta/(m-1)$, together with the new independent variable
\begin{equation}\label{PSvar.large} \eta_f(\xi)=\frac{\alpha}{2m}\int_{(\xi_0+\xi_*)/2}^{\xi}\frac{s}{f^{m-1}(s)}\,ds=\int_{(\xi_0+\xi_*)/2}^{\xi}\frac{ds}{s\mathcal{Z}_f(s)^{\nu}}, \qquad \xi\in (\xi_*,\xi_0).
\end{equation}
Since $m+q>2$, it follows from Lemmas~\ref{lem.p001} and~\ref{lem.p002} that
\begin{equation*}
	\lim_{\xi\searrow \xi_*} \eta_f(\xi) = -\infty, \qquad \lim_{\xi\nearrow\xi_0} \eta_f(\xi) = \infty, %\label{p016}
\end{equation*}
so that $\eta_f$ maps $(\xi_*,\xi_0)$ onto $\mathbb{R}$. Introducing $(Y_f,Z_f)$ defined by
\begin{equation*}
	(\mathcal{Y}_f,\mathcal{Z}_f)=(Y_f\circ\eta_f,Z_f\circ\eta_f) %\label{y001}
\end{equation*}
and setting
\begin{equation*}%\label{interm.K}
	K:=\frac{1}{m}\left[\frac{2m}{\alpha}\right]^{(m-q)/(m-1)}
\end{equation*}
we see that $(Y_f,Z_f)$ solves
\begin{subequations}\label{p017}
\begin{equation}%\label{PSSyst.large}
	\left\{
	\begin{split}
		\frac{dY_f}{d\eta} &= (m-1)Y_f+KZ_f-Y_f^2-2Z_f^{\nu}-NY_fZ_f^{\nu}, \qquad \eta\in\mathbb{R},\\[1mm]
		\frac{dZ_f}{d\eta} &= \mu Y_fZ_f-\frac{2}{\nu} Z_f^{\nu+1},  \qquad \eta\in\mathbb{R}.
	\end{split} \label{p017a}
	\right.
\end{equation}
Observe that we have transformed the non-autonomous second order ordinary differential equation~\eqref{ODE} for $f$ to an autonomous system of ordinary differential equations for $(Y_f,Z_f)$. In addition, $Z_f>0$ in $\mathbb{R}$ and we deduce from Lemmas~\ref{lem.p001} and~\ref{lem.p002} that
\begin{equation}
	\lim_{\eta\to -\infty} (Y_f,Z_f)(\eta) = (m-1,0) \label{p017b}
\end{equation}
and
\begin{equation}
	\lim_{\eta\to \infty} (Y_f,Z_f)(\eta) = (0,0). \label{p017c}
\end{equation}
\end{subequations}
In other words, $(Y_f,Z_f)$ is an heteroclinic orbit of a two-dimensional autonomous dynamical system connecting the critical points $(m-1,0)$ and $(0,0)$. Thus, instead of constructing $f$ as a solution to~\eqref{ODE} with dead core and interface, we will rather construct $(Y_f,Z_f)$ and the next section is devoted to a detailed analysis of the corresponding dynamical system.

%%%%%%%%%%%%%%%%%%%%
%%%%%%%%%%%%%%%%%%%%
\section{An auxiliary dynamical system}\label{sec.dynamic}
%%%%%%%%%%%%%%%%%%%%
%%%%%%%%%%%%%%%%%%%%

Let $K>0$. For $(y,z)\in \mathbb{R}^2$, we define $\mathbf{R}(y,z) = (R_1,R_2)(y,z)$ by
\begin{align*}
	R_1(y,z) & := (m-1) y + K z - y^2 - (2+Ny) z_+^\nu, \\
	R_2(y,z) & := \mu z \left( y - \frac{2}{m-1} z_+^\nu \right),
\end{align*}
where $z_+ := \max\{z,0\}$. Since $\nu>1$, the vector field $\mathbf{R}$ belongs to $C^1(\mathbb{R}^2,\mathbb{R}^2)$ with
\begin{equation*}
	\partial_y R_1(y,z) = (m-1) - 2 y - N z_+^\nu, \qquad \partial_z R_1(y,z) = K - \nu N y Z_+^{\nu-1},
\end{equation*}
\begin{equation*}
	\partial_y R_2(y,z) = \mu z, \qquad \partial_z R_2(y,z) = \mu y - 2 \frac{\nu+1}{\nu} z_+^\nu,
\end{equation*}
for $(y,z)\in \mathbb{R}^2$. We infer from the Cauchy-Lipschitz theorem that, for any $(Y_0,Z_0)\in \mathbb{R}^2$, there is a unique maximal solution
\begin{equation*}
	(Y,Z) = \Phi(\cdot,Y_0,Z_0) \in C^1\big((\eta^-(Y_0,Z_0),\eta^+(Y_0,Z_0));\mathbb{R}^2 \big)
\end{equation*}
to
\begin{subequations}\label{p018}
	\begin{align}
	\frac{d}{d\eta}(Y,Z) & = \mathbf{R}(Y,Z), \label{p018a}\\
	(Y,Z)(0) & = (Y_0,Z_0), \label{p018b}
	\end{align}
\end{subequations}
with $\eta^-(Y_0,Z_0)<0<\eta^+(Y_0,Z_0)$. It readily follows from~\eqref{p018} and the comparison principle that
\begin{equation}
	\left\{
	\begin{split}
		Z(\eta)>0 \;\;\text{ for }\;\; \eta\in (\eta^-(Y_0,Z_0),\eta^+(Y_0,Z_0)) \;\;\text{ if }\;\; Z_0>0, \\
		Z(\eta)<0 \;\;\text{ for }\;\; \eta\in (\eta^-(Y_0,Z_0),\eta^+(Y_0,Z_0)) \;\;\text{ if }\;\; Z_0<0.
	\end{split} \label{p019}
	\right.
\end{equation}
We next observe that the vector field $\mathbf{R}$ has three critical points
\begin{equation*}
	P_0=(0,0), \qquad P_1=(m-1,0), \qquad P_2= \left(\frac{2Z_K^\nu}{m-1},Z_K \right),
\end{equation*}
where
\begin{equation}\label{xp2}
	Z_K := \left( \frac{K(m-1)^2}{2[N(m-1)+2]}\right)^{\mu/(m-q)}.
\end{equation}
According to~\eqref{p017} and the discussion at the end of Section~\ref{subsec.aaf}, we are interested in complete orbits of~\eqref{p018} with a positive second component and connecting $P_1$ to $P_0$, which are thus the two critical points of interest. Let us start with the analysis of the local behavior near $P_1$.

%%%%%%%%%%%%%%%%%%%%
%%%%%%%%%%%%%%%%%%%%
\subsection{Local behavior near $P_1$}\label{subsec.lb1}
%%%%%%%%%%%%%%%%%%%%
%%%%%%%%%%%%%%%%%%%%

We define the unstable manifold $W_u(P_1)$ as usual by
\begin{equation*}
	W_u(P_1) := \left\{ (Y_0,Z_0)\in\mathbb{R}^2\ :\ \lim_{\eta\to -\infty} \Phi(\eta,Y_0,Z_0) = P_1 \right\}.
\end{equation*}{\tiny }

%%%%%%%%%%%%%%%%%%%%
\begin{lemma}\label{lem.P1}
	The critical point $P_1$ is a saddle point and the intersection of its unstable manifold $W_u(P_1)$ and the positive cone $(0,\infty)^2$ of $\mathbb{R}^2$ is a single orbit
	\begin{equation*}
		l_1(K) = W_u(P_1)\cap (0,\infty)^2 = \{(Y_{1,K},Z_{1,K})(\eta)\ :\ \eta\in (-\infty,\eta_{1,K}^+)\}
	\end{equation*}
	of~\eqref{p018} with $\eta_{1,K}^+ \in (0,\infty]$. Moreover,
	\begin{equation*}
		Z_{1,K}^\nu \in L^1(-\infty,0) \;\;\text{ and }\;\; \lim_{\eta\to-\infty} \frac{Y_{1,K}(\eta)+1-m}{Z_{1,K}(\eta)} = \frac{K}{(m+q-1)(m-1)}.
	\end{equation*}
\end{lemma}
%%%%%%%%%%%%%%%%%%%%

\begin{proof}
	Since
	\begin{equation*}
		M_1 = D\mathbf{R}(P_1)=
		\begin{pmatrix}
			-(m-1) & K \\
			0 & \mu(m-1)
		\end{pmatrix}
	\end{equation*}
	has two eigenvalues $\lambda_1=-(m-1)<0<\lambda_2=\mu(m-1)>0$ with associated eigenvectors $\mathbf{e}_1=(1,0)$ and $\mathbf{e}_2=(K,(m+q-1)(m-1))$, the critical point~$P_1$ is a saddle point and thus an hyperbolic critical point. According to the stable manifold theorem, see \cite[Theorem~19.11]{Amann} for instance, there are an open neighborhood $\mathcal{U}_1$ of $P_1$ in $\mathbb{R}^2$, an open interval $\mathcal{J}_1$ containing $0$, and $h_1\in C^1(\mathcal{J}_1;\mathcal{U}_1)$ such that $h_1(0) = h_1'(0)=0$, and
	\begin{equation*}
		\left\{ P_1 + \zeta \mathbf{e}_2 + h_1(\zeta) \mathbf{e}_1\ :\ \zeta\in (-\delta_1,\delta_1) \right\} \subset W_u^{\mathcal{U}_1}(P_1) \subset \left\{ P_1 + \zeta \mathbf{e}_2 + h_1(\zeta) \mathbf{e}_1\ :\ \zeta\in \mathcal{J}_1 \right\}
	\end{equation*}
	where
	\begin{equation*}
		W_u^{\mathcal{U}_1}(P_1) := \big\{ (Y_0,Z_0)\in W_u(P_1)\ :\ \Phi(\eta,Y_0,Z_0)\in \mathcal{U}_1 \;\text{ for all }\; \eta\in (-\infty,0) \big\}.
	\end{equation*}
	We first observe that, since $h_1(0)=h_1'(0)=0$, $P_1+\zeta \mathbf{e}_2 + h_1(\zeta) \mathbf{e}_1 = \big( m-1+K\zeta+h_1(\zeta),(m+q-1)(m-1)\zeta \big)$ belongs to $(0,\infty)^2$ for $\zeta>0$ sufficiently small. Therefore,
	\begin{equation*}
		W_u^{\mathcal{U}_1}(P_1) \cap (0,\infty)^2 \ne\emptyset.
	\end{equation*}
	Consider next $(Y_0,Z_0)\in W_u(P_1)\cap (0,\infty)^2$ and set $(Y,Z)=\Phi(\cdot,Y_0,Z_0)$. We infer from~\eqref{p018} and~\eqref{p019} that $Z(\eta)>0$ for $\eta\in (-\infty,0)$ and
	\begin{equation*}
		\frac{d}{d\eta}\ln{Z(\eta)} = \mu \left( Y(\eta) - \frac{2}{m-1} Z^\nu(\eta) \right), \qquad \eta\in (-\infty,0).
	\end{equation*}
	Since $(Y_0,Z_0)\in W_u(P_1)$, we deduce from the above identity that
	\begin{equation*}
		\lim_{\eta\to -\infty} \frac{d}{d\eta}\ln{Z(\eta)}  = \mu(m-1)>0.
	\end{equation*}
	Consequently, there is $\eta_1\in (-\infty,0)$ such that, for $\eta\in (-\infty,\eta_1)$,
	\begin{equation*}
		\frac{\mu(m-1)}{2} \le \frac{d}{d\eta}\ln{Z(\eta)}.
	\end{equation*}
	Hence, after integration,
	\begin{equation*}
		Z^\nu(\eta) \le Z^\nu(\eta_1) e^{(m-1)^2(\eta-\eta_1)/2}, \qquad \eta\in (-\infty,\eta_1),
	\end{equation*}
	from which we conclude that
	\begin{equation*}
		Z^\nu \in L^1(-\infty,0). %\label{p020}
	\end{equation*}
	
	Let us finally consider $(Y_0,Z_0)\in W_u(P_1)\cap (0,\infty)^2$ and $(\tilde{Y}_0,\tilde{Z}_0)\in W_u(P_1)\cap (0,\infty)^2$. There is $\eta_0<0$ such that
	\begin{align*}
		(Y,Z)(\eta) & := \Phi(\eta,Y_0,Z_0) \in W_u^{\mathcal{U}_1}(P_1)\cap (0,\infty)^2, \qquad \eta\in (-\infty,\eta_0), \\
		(\tilde{Y},\tilde{Z})(\eta) & := \Phi(\eta,\tilde{Y}_0,\tilde{Z}_0) \in W_u^{\mathcal{U}_1}(P_1)\cap (0,\infty)^2, \qquad \eta\in (-\infty,\eta_0).
	\end{align*}
	There are then $(\zeta_0,\tilde{\zeta}_0)\in \mathcal{J}_1^2$ such that
	\begin{equation*}
		(Y,Z)(\eta_0) = P_1 + \zeta_0 \mathbf{e}_2 + h_1(\zeta_0) \mathbf{e}_1 \;\;\text{ and }\;\; (\tilde{Y},\tilde{Z})(\eta_0) = P_1 + \tilde{\zeta}_0 \mathbf{e}_2 + h_1(\tilde{\zeta}_0) \mathbf{e}_1,
	\end{equation*}
	and we may assume that $\zeta_0 \le \tilde{\zeta}_0$ without loss of generality. Since $\big( \tilde{Y},\tilde{Z} \big)(\cdot+\eta_0)$ belongs to $W_u^{\mathcal{U}_1}(P_1)$, there is $\tilde{\eta}_0\le \eta_0$ such that
	\begin{equation*}
		\big( \tilde{Y}, \tilde{Z} \big)(\tilde{\eta}_0) = P_1 + \zeta_0\mathbf{e}_2 + h_1(\zeta_0) \mathbf{e}_1 = (Y,Z)(\eta_0).
	\end{equation*}
	Consequently, $\big( \tilde{Y},\tilde{Z} \big)(\cdot,\tilde{\eta}_0-\eta_0) = (Y,Z)$ and $W_u(P_1)\cap (0,\infty)^2$ contains a single orbit $(Y_{1,K},Z_{1,K})$ which can be assumed to be defined on $(-\infty,\eta_{1,K}^+)$ for some $\eta_{1,K}^+ \in (0,\infty]$.
	
	Finally, since $(Y_{1,K},Z_{1,K})$ belongs to $W_u(P_1)\cap (0,\infty)^2$, there are $\bar{\eta}<0$ and $\bar{\zeta}\in C((-\infty,\bar{\eta}))$ such that, for any $\eta\in (-\infty,\bar{\eta})$, $\zeta(\eta)\in (-\delta_1,\delta_1)$ and
	\begin{align*}
		(Y_{1,K},Z_{1,K})(\eta) & = P_1 + \bar{\zeta} \mathbf{e}_2 + h_1(\bar{\zeta}(\eta)) \mathbf{e}_1 \\
		& = \big( m-1+K \bar{\zeta}(\eta) + h_1(\bar{\zeta}(\eta)),(m+q-1)(m-1) \bar{\zeta}(\eta) \big).
	\end{align*}
	Hence, recalling that $h_1(0) = h_1'(0)=0$,
	\begin{equation*}
		\lim_{\eta\to-\infty} \frac{Y_{1,K}(\eta)+1-m}{Z_{1,K}(\eta)} = \lim_{\eta\to -\infty} \frac{K \bar{\zeta}(\eta) + h_1(\bar{\zeta}(\eta))}{(m+q-1)(m-1) \bar{\zeta}(\eta)} = \frac{K}{(m+q-1)(m-1)},
	\end{equation*}
	and the proof is complete.
\end{proof}

%%%%%%%%%%%%%%%%%%%%
%%%%%%%%%%%%%%%%%%%%
\subsection{Local behavior near $P_0$}\label{subsec.lb2}
%%%%%%%%%%%%%%%%%%%%
%%%%%%%%%%%%%%%%%%%%

We next turn to the local analysis in a neighborhood of $P_0$.

%%%%%%%%%%%%%%%%%%%%
\begin{lemma}\label{lem.P0}
	The critical point $P_0$ has a one-dimensional center manifold $W_c(P_0)$ and a one-dimensional unstable manifold $W_u(P_0) = (-\infty,m-1)\times \{0\}$. Moreover, $W_c(P_0)\cap \big( \mathbb{R}\times (0,\infty) \big)$ is uniquely determined and contains a single orbit
	\begin{equation*}
		l_0(K) = W_c(P_0)\cap \big( \mathbb{R}\times (0,\infty) \big) = \{(Y_{0,K},Z_{0,K})(\eta)\ :\ \eta\in (\eta_{0,K}^-,\infty)\}
	\end{equation*}
	of~\eqref{p018} with $\eta_{0,K}^- \in [-\infty,0)$. Moreover,
	\begin{equation*}
		Z_{0,K}^\nu \in L^1(0,\infty), \quad \lim_{\eta\to\infty} Z_{0,K}(\eta)=0, \quad \lim_{\eta\to\infty} \frac{Y_{0,K}(\eta)}{Z_{0,K}(\eta)} = -\frac{K}{m-1}.
	\end{equation*}
\end{lemma}
%%%%%%%%%%%%%%%%%%%%

\begin{proof}
	We first observe that
	\begin{equation*}
		M_0 = D\mathbf{R}(P_0)=
		\begin{pmatrix}
			m-1 & K \\
			0 & 0
		\end{pmatrix}
	\end{equation*}
	has two eigenvalues $\lambda_1=m-1>0$ and $\lambda_2=0$ with associated eigenvectors $\mathbf{e}_1=(1,0)$ and $\mathbf{e}_2=(-K,m-1)$. Therefore, by the center manifold theorem, see \cite[Section~2.7]{Pe}, there exists a $C^1$-smooth one-dimensional center manifold $W_c(P_0)$ and a one-dimensional unstable manifold $W_u(P_0)$ for the flow $\Phi$ associated to~\eqref{p018}.
	
	On the one hand, according to \cite[Definition~1.1 \& Section~2]{Sij}, there are $\delta_0>0$ and $h_0\in C^1(-\delta_0,\delta_0)$ satisfying $h_0(0)=h_0'(0)=0$ such that
	\begin{equation}
		W_c(P_0) = \left\{ \frac{\zeta}{m-1} \mathbf{e}_2 + h_0(\zeta) \mathbf{e}_1\ :\ \zeta\in (-\delta_0,\delta_0) \right\} \label{p021}
	\end{equation}
	and we may assume without loss of generality that
	\begin{equation}
		|h_0(\zeta)| \le \frac{K}{2(m-1)} |\zeta|, \qquad \zeta\in (-\delta_0,\delta_0), \label{p024}
	\end{equation}
	due to $h_0(0)=0$. Consider now $(Y_0,Z_0)\in W_c(P_0)\cap \big( \mathbb{R}\times (0,\infty) \big)$. Since both $W_c(P_0)$ and $\mathbb{R}\times (0,\infty)$ are positively invariant for the flow $\Phi$ associated to~\eqref{p018}, we obtain that
	\begin{equation}
		(Y,Z)(\eta) = \Phi(\eta,Y_0,Z_0) \in W_c(P_0)\cap \big( \mathbb{R}\times (0,\infty) \big), \qquad \eta\in [0,\eta^+(Y_0,Z_0)). \label{p022}
	\end{equation}
	We then infer from~\eqref{p021} and~\eqref{p022} that
	\begin{equation}
		Y(\eta) = - \frac{K}{m-1} Z(\eta) + h_0(Z(\eta)), \quad 0<Z(\eta)<\delta_0, \qquad \eta\in [0,\eta^+(Y_0,Z_0)), \label{p023}
	\end{equation}
	which readily implies that
	\begin{equation}
		 \eta^+(Y_0,Z_0)=\infty \;\;\text{ and }\;\; \lim_{\eta\to\infty} \frac{Y(\eta)}{Z(\eta)} = - \frac{K}{m-1}. \label{p025}
	\end{equation}
	Moreover, by~\eqref{p018}, \eqref{p024}, and~\eqref{p023},
	\begin{equation*}
		\frac{dZ}{d\eta} = - \frac{\mu K}{m-1} Z^2 + \mu Z h_0(Z) - \frac{2\mu}{m-1} Z^{\nu+1} \le - \frac{\mu K}{2(m-1)} Z^2, %\label{p026}
	\end{equation*}
	from which we deduce that $Z$ is decreasing on $[0,\infty)$ and satisfies
	\begin{equation}
		\lim_{\eta\to \infty} Z(\eta) = 0, \qquad \lim_{\eta\to\infty} \left( \frac{1}{Z^2(\eta)} \frac{dZ}{d\eta}(\eta) \right) = - \frac{\mu K}{m-1}, \label{p027}
	\end{equation}
	using as well the positivity~\eqref{p023} of $Z$, $h_0(0)=h_0'(0)=0$, and $\nu>1$. It follows in particular from~\eqref{p027} that $Z(\eta)\sim (m-1)/(\mu K \eta)$ as $\eta\to\infty$ and thus that
	\begin{equation}
		Z^\nu\in L^1(0,\infty) \label{p028}
	\end{equation}
	since $\nu>1$. We next argue as in the proof of Lemma~\ref{lem.P1} to show that $W_c(P_0)$ contains a single orbit which satisfies the properties listed in Lemma~\ref{lem.P0} according to~\eqref{p025}, \eqref{p027}, and~\eqref{p028}.
	
	On the other hand, assume that there are $\tilde{\delta}_0>0$ and $\tilde{h}_0\in C^1(-\tilde{\delta}_0,\tilde{\delta}_0)$ satisfying $\tilde{h}_0(0)=\tilde{h}_0'(0)=0$ such that
	\begin{equation*}
		\tilde{W}_c(P_0) = \left\{ \frac{\zeta}{m-1} \mathbf{e}_2 + \tilde{h}_0(\zeta) \mathbf{e}_1\ :\ \zeta\in (-\tilde{\delta}_0,\tilde{\delta}_0) \right\}
	\end{equation*}
	is a $C^1$-smooth center manifold for the flow $\Phi$ associated to~\eqref{p018}. Pick $(Y_0,Z_0)\in W_c(P_0)\cap \big( \mathbb{R}\times (0,\infty) \big)$ such that $Z_0<\tilde{\delta}_0$ and set $(Y,Z) = \Phi(\cdot,Y_0,Z_0)$. Since $\eta^+(Y_0,Z_0)=\infty$ and $Z(\eta)$ converges to zero as $\eta\to\infty$ according to~\eqref{p025} and~\eqref{p027}, the condition \cite[(3.20)]{Sij} is satisfied and an application of \cite[Theorem~3.2'(i)]{Sij} ensures that $h_0(Z(\eta)) = \tilde{h}_0(Z(\eta))$ for $\eta\in [0,\infty)$. Therefore, $h_0(Z_0) = \tilde{h}_0(Z_0)$, so that $(Y_0,Z_0)$ belongs to $\tilde{W}_c(P_0)\cap \big( \mathbb{R}\times (0,\infty) \big)$. We have thus established that
	\begin{equation*}
		W_c(P_0)\cap \big( \mathbb{R}\times (0,\infty) \big) \subset \tilde{W}_c(P_0)\cap \big( \mathbb{R}\times (0,\infty) \big).
	\end{equation*}
	A similar argument gives the reverse inclusion, thereby establishing the uniqueness of the center manifold.
	
	We finally observe that, for $Y_0\in\mathbb{R}$, the solution $(Y,Z)=\Phi(\cdot,Y_0,0)$ to~\eqref{p018} is explicit and given by
	\begin{equation*}
		Y(\eta) = \frac{(m-1) Y_0}{Y_0 + [(m-1)-Y_0] e^{-(m-1)\eta}} \;\;\text{ and }\;\;  Z(\eta)  = 0
	\end{equation*}
	for $\eta\in (-\infty,\eta^+(Y_0,0))$ with
	\begin{align*}
		\eta^+(Y_0,0) & = \infty \;\;\text{ for }\;\; Y_0\in (-\infty,m-1], \\
		\eta^+(Y_0,0) & = \frac{1}{m-1} \ln{\left( \frac{Y_0+1-m}{Y_0} \right)} \;\;\text{ for }\;\; Y_0>m-1.
	\end{align*}
	Moreover, for $Y_0\in (-\infty,m-1)$,
	\begin{equation*}
		\lim_{\eta\to-\infty} \Phi(\eta,Y_0,0) = P_0, \qquad \lim_{\eta\to \infty} \Phi(\eta,Y_0,0) = P_1,
	\end{equation*}
	so that $(-\infty,m-1)\times\{0\} \subset W_u(P_0)$. Since $W_u(P_0)$ is one-dimensional and unique, we conclude that $W_u(P_0)=(-\infty,m-1)\times\{0\}$.
\end{proof}

%%%%%%%%%%%%%%%%%%%%
%%%%%%%%%%%%%%%%%%%%
\section{Invariant regions and trajectory analysis}\label{subsec.aads}
%%%%%%%%%%%%%%%%%%%%
%%%%%%%%%%%%%%%%%%%%

To proceed further and derive more qualitative properties of solutions to~\eqref{p018}, it turns out to be more convenient to perform a transverse change of variables and set, for $K>0$ and $(Y_0,Z_0)\in \mathbb{R}^2$,
\begin{equation}
	U(\lambda_K\eta) = \frac{Z(\eta)}{Z_K}, \quad V(\lambda_K\eta) = \frac{(m-1)Y(\eta) - 2 Z_+^\nu(\eta)}{\lambda_K} \label{p029}
\end{equation}
for $\eta\in (\eta^-(Y_0,Z_0),\eta^+(Y_0,Z_0))$, where $Z_K$ is defined in~\eqref{xp2}, $(Y,Z)=\Phi(\cdot,Y_0,Z_0)$, and
\begin{equation*}
	\lambda_K = \sqrt{\frac{2(mN-N+2)}{m-1}}Z_K^\nu.
\end{equation*}
Introducing
\begin{equation*}
	(U_0,V_0) = \left( \frac{Z_0}{Z_K} , \frac{(m-1)Y_0 - 2 Z_0^\nu}{\lambda_K} \right), \quad \zeta=\lambda_K\eta,
\end{equation*}
the system of ordinary differential equations solved by $(U,V)$ reads
\begin{equation}
	\frac{d}{d\zeta}(U,V) = \mathbf{S}_K(U,V), \qquad (U,V)(0) = (U_0,V_0), \label{p030}
\end{equation}
where the vector field $\mathbf{S}_K:=(S_{1,K},S_{2,K})$ is given by
\begin{equation*}
	S_{1,K}(u,v) = \frac{uv}{\nu}, \qquad S_{2,K}(u,v) = C_1(K) v - \frac{v^2}{m-1} + u - u_+^{2\nu} - C_2 v u_+^\nu
\end{equation*}
for $(u,v)\in\mathbb{R}^2$ and
\begin{equation*}
	C_1(K) := \frac{m-1}{\lambda_K} , \qquad C_2 := \frac{N(m-1)+ 2m + 2}{\sqrt{2(m-1)[N(m-1)+2]}}. %\label{coefs}
\end{equation*}
We point out that
\begin{equation}
	C_1(K) = \frac{\sqrt{m-1}}{K^{(m-1)/(m-q)}} \left( \frac{2 [N(m-1)+2]}{(m-1)^2} \right)^{\mu/2(m-q)} = \frac{C_1(1)}{K^{(m-1)/(m-q)}} \label{p031}
\end{equation}
is a decreasing function of the parameter $K$.

Since $\nu>1$, the vector field $\mathbf{S}_K$ belongs to $C^1(\mathbb{R}^2;\mathbb{R}^2)$ with
\begin{equation*}
	\partial_u S_{1,K}(u,v) = \frac{v}{\nu}, \qquad \partial_v S_{1,K}(u,v) = \frac{u}{\nu},
\end{equation*}
\begin{equation*}
	\partial_u S_{2,K}(u,v) = 1 - 2 \nu u_+^{2\nu-1} - \nu C_2 v u_+^{\nu-1}, \qquad \partial_v S_{2,K}(u,v) = C_1(K) - \frac{2v}{m-1} - C_2 u_+^\nu,
\end{equation*}
and it has three critical points
\begin{equation*}
	Q_0 := (0,0), \quad Q_{1,K} := \big( 0,(m-1)C_1(K) \big), \quad Q_2 := (1,0),
\end{equation*}
corresponding to the critical points $(P_0,P_1,P_2)$ of $\mathbf{R}$ in~\eqref{p018}. It follows from the Cauchy-Lipschitz theorem that, for any $(U_0,V_0)\in\mathbb{R}^2$, there is a unique maximal solution
\begin{equation*}
	(U,V) = \Psi_K(\cdot,U_0,V_0) \in C^1\big((\zeta_K^-(U_0,V_0),\zeta_K^+(U_0,V_0));\mathbb{R}^2 \big)
\end{equation*}
to~\eqref{p030} with $\zeta_K^-(U_0,V_0)<0<\zeta_K^+(U_0,V_0)$. We also introduce a parametrization of the orbits $l_0(K)$ and $l_1(K)$ in the $(u,v)$-plane and set
\begin{equation}
	\begin{split}
		l_0(K) & = \big\{ (U_{0,K},V_{0,K})(\zeta)\ :\ \zeta\in (\zeta_{0,K}^-,\infty) \big\}, \\
		l_1(K) & = \big\{ (U_{1,K},V_{1,K})(\zeta)\ :\ \zeta\in (-\infty,\zeta_{1,K}^+) \big\},
	\end{split} \label{p049}
\end{equation}
with $\zeta_{0,K}^-\in [-\infty,0)$ and $\zeta_{1,K}^+\in (0,\infty]$. By Lemma~\ref{lem.P1}, Lemma~\ref{lem.P0}, and~\eqref{p029},
\begin{equation}
	\lim_{\zeta\to \infty} U_{0,K}(\zeta) = \lim_{\zeta\to \infty} V_{0,K}(\zeta) = 0, \qquad \lim_{\zeta\to\infty} \frac{V_{0,K}}{U_{0,K}}(\zeta) = - \frac{1}{C_1(K)} \label{p050}
\end{equation}
and
\begin{equation*}
	\lim_{\zeta\to -\infty} U_{1,K}(\zeta) = 0, \qquad \lim_{\zeta\to -\infty} V_{1,K}(\zeta) = (m-1)C_1(K). %\label{p051}
\end{equation*}
We are now in a position to complete the local analysis of the critical point $Q_2$, which was not considered in the previous sections.

%%%%%%%%%%%%%%%%%%%%
\begin{lemma}\label{lem.Q2}
	There are positive constants $K_u<K_f<K_s$ depending only on $N$, $m$, and $q$ such that the critical point $Q_2$ is
	\begin{itemize}
		\item an \emph{unstable node} (that is, two positive eigenvalues) for $K\in(0,K_u]$.
		\item an \emph{unstable focus} (that is, two complex conjugate eigenvalues with positive real parts) for $K\in(K_u,K_f)$.
		\item a \emph{stable focus} (that is, two complex conjugate eigenvalues with negative real parts) for $K\in(K_f,K_s)$.
		\item a \emph{stable node} (that is, two negative eigenvalues) for $K\in[K_s,\infty)$.
	\end{itemize}
\end{lemma}
%%%%%%%%%%%%%%%%%%%%

\begin{proof}
	Since
	\begin{equation*}
		D\mathbf{S}_K(Q_2) =
		\begin{pmatrix}
			0 & \displaystyle{\frac{1}{\nu}} \\
			\displaystyle{- \frac{m-q}{\mu}} & C_1(K) - C_2
		\end{pmatrix}
	\end{equation*}
	the eigenvalues $(\lambda_1,\lambda_2)$ of $D\mathbf{S}_K(Q_2)$ are solutions to
	\begin{equation*}
		\lambda^2 - (C_1(K)-C_2) \lambda + (m-q)/(m-1)=0.
	\end{equation*}
	Its discriminant being $\big( C_1(K)-C_2 \big)^2 - 4(m-q)/(m-1)$, it readily follows from the monotonicity~\eqref{p031} of $K\mapsto C_1(K)$ that there are $0<K_u<K_s$ such that
	\begin{align*}
		C_2 - C_1(K) & \ge 2 \sqrt{\frac{m-q}{m-1}}, \qquad K\ge K_s, \\
		C_1(K) - C_2 & \ge 2 \sqrt{\frac{m-q}{m-1}}, \qquad 0<K\le K_u, \\
		|C_1(K)-C_2| & < 2 \sqrt{\frac{m-q}{m-1}}, \qquad K\in (K_u,K_s).
	\end{align*}
	Then $(\lambda_1,\lambda_2)\in (0,\infty)^2$ when $K\ge K_s$, $(\lambda_1,\lambda_2)\in (-\infty,0)^2$ when $K\le K_u$ and there is a unique $K_f\in (K_u,K_s)$ (which is characterized by the identity $C_1(K_f)=C_2$) such that $(\lambda_1,\lambda_2)$ are complex conjugates with positive real parts for $K\in (K_u,K_f)$ and complex conjugates with negative real parts for $K\in (K_f,K_s)$.
\end{proof}

%%%%%%%%%%%%%%%%%%%%
%%%%%%%%%%%%%%%%%%%%
\subsection{Configuration of the phase plane for small $K>0$} \label{subsec.cppsK}
%%%%%%%%%%%%%%%%%%%%
%%%%%%%%%%%%%%%%%%%%

We provide in the next result further information on the orbits $l_0(K)$ and $l_1(K)$ for small values of $K$ (corresponding to large values of the self-similar exponents $\alpha$ and $\beta$).

%%%%%%%%%%%%%%%%%%%%
\begin{proposition}\label{prop.Ksmall}
	There is $K_\star>0$ such that the following two statements hold true for any $K\in(0,K_\star)$:
	\begin{enumerate}
		\item The orbit $l_0(K)$ is a connection between $Q_2$ and $Q_0$ and is completely contained in $(0,\infty)\times(-\infty,0)$.
		\item The orbit $l_1(K)$ crosses the half-line $(0,\infty)\times \{0\}$ at a finite point.
	\end{enumerate}
\end{proposition}
%%%%%%%%%%%%%%%%%%%%

\begin{proof}
	\noindent\textbf{Step~1.} Let us introduce the triangle
	\begin{equation*}
		\mathcal{T}_1 := \{ (u,v)\in (0,1)\times\mathbb{R}\ :\ u-1 \le v \le 0\}.
	\end{equation*}
	We shall show that $\mathcal{T}_1$ is negatively invariant for the flow associated to~\eqref{p030} when $K$ is small. Indeed, we first recall that the first equation in~\eqref{p030} ensures that $(0,\infty)\times\mathbb{R}$ is invariant for~\eqref{p030}. Next, on $(0,1)\times \{0\}$,
	\begin{equation*}
		\big\langle D\mathbf{S}_K(u,0) , (0,1) \big\rangle = u - u^{2\nu} >0,
	\end{equation*}
	an inequality which implies that $(0,1)\times (-\infty,0]$ is negatively invariant for~\eqref{p030}.
	Finally, for $u\in (0,1)$,
	\begin{align*}
		\big\langle D\mathbf{S}_K(u,u-1) , (1,-1) \big\rangle&  = (1-u) \left( C_1(K) - \frac{u}{\nu} + \frac{1-u}{m-1} - C_2 u^\nu \right) - u + u^{2\nu} \\
		& \ge(1-u) \left( C_1(K) - \frac{1}{\nu} - C_2 \right) + (2\nu-1) u \int_1^u u_*^{2\nu-2}\ du_* \\
		& \ge(1-u) \left( C_1(K) - \frac{1}{\nu} - C_2 - (2\nu-1) \right),
	\end{align*}
	and the right-hand side of the above inequality is positive for $K$ sufficiently small according to~\eqref{p031}. Therefore, there is $K_\star>0$ such that $\{(u,v)\in (0,1)\times\mathbb{R}\ :\ u-1\le v\}$ is negatively invariant for~\eqref{p030} when $K\in (0,K_\star)$. We have thus established that
	\begin{equation}
		\mathcal{T}_1 \;\text{ is negatively invariant for~\eqref{p030} when }\; K\in (0,K_\star). \label{p032}
	\end{equation}

\medskip
	
\noindent\textbf{Step~2.} Let $K\in (0,K_\star)$. According to~\eqref{p050} and Lemma~\ref{lem.P0}, $U_{0,K}(\zeta)>0$ for $\zeta>0$ and there is $\zeta_\infty>0$ such that, for $\zeta>\zeta_\infty$,
\begin{equation*}
	\max\{ U_{0,K}(\zeta) , -V_{0,K}(\zeta) \})<\frac{1}{2}, \quad -\frac{2}{C_1(K)} < \frac{V_{0,K}(\zeta)}{U_{0,K}(\zeta)} < - \frac{1}{2 C_1(K)}
\end{equation*}
Therefore, for $\zeta>\zeta_\infty$, $V_{0,K}(\zeta) < -U_{0,K}(\zeta)/2C_1(K)< 0 < U_{0,K}(\zeta)$ and
\begin{equation*}
	U_{0,K}(\zeta) - 1 < - \frac{1}{2} < V_{0,K}(\zeta);
\end{equation*}
that is, $(U_{0,K},V_{0,K})(\zeta) \in \mathcal{T}_1$ for $\zeta>\zeta_\infty$. Owing to the negative invariance~\eqref{p032} of $\mathcal{T}_1$, we conclude that $(U_{0,K},V_{0,K})(\zeta)\in \mathcal{T}_1$ for all $\zeta\in (\zeta_{0,K}^-,\infty)$. This property, along with~\eqref{p030} and the boundedness of $\mathcal{T}_1$, implies in particular that $\zeta_{0,K}^-=-\infty$ and
\begin{equation}
	(U_{0,K},V_{0,K})(\zeta) \in \mathcal{T}_1 \;\;\text{ with }\;\; \frac{dU_{0,K}}{d\zeta}(\zeta) < 0, \qquad \zeta\in\mathbb{R}. \label{p033}
\end{equation}
An immediate consequence of~\eqref{p033} is the existence of $U_{-\infty}\in (U_{0,K}(0),1]$ such that
\begin{equation}
	\lim_{\zeta\to -\infty} U_{0,K}(\zeta) = U_{-\infty}. \label{p080}
\end{equation}
We now study the behavior of $(U_{0,K},V_{0,K})(\zeta)$ as $\zeta\to-\infty$ and introduce the corresponding $\alpha$-limit set
\begin{equation*}
	\alpha_{0,K} := \{ \text{cluster points in }\; \mathbb{R}^2 \;\text{ of }\; (U_{0,K},V_{0,K})(\zeta) \;\text{ as }\; \zeta\to -\infty \}.
\end{equation*}
%Clearly, thanks to~\eqref{p033},
%\begin{equation}
%	\alpha_{0,K} \subset \overline{\mathcal{T}}_1 = \{(u,v)\in [0,1]\times \mathbb{R}\ :\ u-1 \le v \le 0\}. \label{p034}
%\end{equation}
%Assume first for contradiction that $\alpha_{0,K}$ contains a non-critical periodic orbit $\gamma$. By \cite[Remark~24.13(b)]{Amann} and~\eqref{p034}, $\gamma$ is a $C^2$-Jordan curve contained in the simply connected set $\overline{\mathcal{T}}_1$, so that its interior $\Omega_\gamma$ is also included in $\overline{\mathcal{T}}_1$. We then infer from \cite[Corollary~24.22]{Amann} that there is at least a critical point of~\eqref{p030} in $\Omega_\gamma$. However, there is no critical point of~\eqref{p030} in the interior of $\overline{\mathcal{T}}_1$, and a contradiction. Therefore, there is no non-critical periodic orbit in $\alpha_{0,K}$ and we infer from the generalized Poincar\'e-Bendixson theorem, see \cite[Section~3.7, Theorem~2]{Pe}, that $\alpha_{0,K}$ is either reduced to a single critical point of~\eqref{p030} or consists of a finite number of critical points of~\eqref{p030} and a countable number of complete orbits connecting them. Now, \eqref{p034} entails that only $Q_0$ and $Q_2$ may belong to $\alpha_{0,K}$, while \eqref{p033} implies that the first component of any critical point in $\alpha_{0,K}$ has to be positive. Therefore, $Q_0\not\in \alpha_{0,K}$ and we conclude that $\alpha_{0,K}=\{Q_2\}$, thereby completing the proof of the first statement of Proposition~\ref{prop.Ksmall}.
Assume for contradiction that $\alpha_{0,K}$ contains no critical point of~\eqref{p030}. According to the Poincar\'e-Bendixson theorem, see \cite[Section~3.7, Theorem~1]{Pe}, $\alpha_{0,K}$ is then a periodic orbit $(U_P,V_P)$ of~\eqref{p030} with period $P>0$ and it satisfies $U_P(\zeta) = U_{-\infty}$ for all $\zeta\in [0,P]$ by~\eqref{p080}. This property, along with~\eqref{p030}, implies that $V_P$ solves
\begin{equation*}
	\frac{dV_P}{d\zeta} = C_1(K) V_P - \frac{V_P^2}{m-1} - C_2 U_{-\infty}^\nu V_P + U_{-\infty} - U_{-\infty}^{2\nu}, \qquad \zeta\in [0,P].
\end{equation*}
Introducing
\begin{equation*}
	H(v) := - \big( U_{-\infty} - U_{-\infty}^{2\nu} \big) v - \frac{C_1(K)-C_2 U_{-\infty}^\nu}{2} v^2 + \frac{v^3}{3(m-1)}, \qquad v\in\mathbb{R},
\end{equation*}
we see that, for $\zeta\in [0,P]$,
\begin{equation*}
	\frac{d}{d\zeta} H(V_P) = - \left| \frac{dV_P}{d\zeta} \right|^2.
\end{equation*}
Hence, after integration over $[0,P]$,
\begin{equation*}
	0 = H(V_P(P)) - H(V_P(0)) = - \int_0^P \left| \frac{dV_P}{d\zeta}(\zeta_*) \right|^2\ d\zeta_*,
\end{equation*}
from which we conclude that there is $V_{-\infty}\le 0$ such that $V_P\equiv V_\infty$ on $[0,P]$, the non-positivity of $V_\infty$ being a consequence of~\eqref{p032}. The periodic orbit $(U_P,V_P)$ then reduces to a critical point $(U_{-\infty},V_{-\infty})$ and a contradiction. We have thus established that $\alpha_{0,K}$ contains at least one critical point of~\eqref{p030}, which belongs to $\overline{\mathcal{T}_1}$ and with first component $U_{-\infty}>0$ by~\eqref{p080}. Therefore, $Q_0\not\in \alpha_{0,K}$ and we conclude that $\alpha_{0,K}=\{Q_2\}$, thereby completing the proof of the first statement of Proposition~\ref{prop.Ksmall}.

\medskip

\noindent\textbf{Step~3.} We finally prove the second statement of Proposition~\ref{prop.Ksmall}. To this end, we assume for contradiction that there is $K\in (0,K_u)$ such that
\begin{equation}
	l_1(K) \subset (0,\infty)^2. \label{p035}
\end{equation}
By~\eqref{p030}, \eqref{p035}, and Young's inequality,
\begin{align*}
	\frac{dV_{1,K}}{d\zeta} & \le \frac{m-1}{2} C_1(K)^2+ \frac{V_{1,K}^2}{2(m-1)} - \frac{V_{1,K}^2}{m-1} + \frac{U_{1,K}^{2\nu}}{2\nu} + \frac{2\nu-1}{2\nu} - U_{1,K}^{2\nu} \\
	& \le \frac{m-1}{2} C_1(K)^2 + \frac{2\nu-1}{2\nu} - \frac{V_{1,K}^2}{2(m-1)} - \frac{2\nu-1}{2\nu} U_{1,K}^{2\nu}
\end{align*}
and
\begin{align*}
	\frac{dU_{1,K}}{d\zeta} & \le \frac{U_{1,K}^{2\nu}}{2\nu} + \frac{2\nu-1}{2\nu} \left( \frac{V_{1,K}}{\nu} \right)^{(2\nu-1)/2\nu} \le \frac{U_{1,K}^{2\nu}}{2\nu} + \frac{V_{1,K}^2}{4(m-1)} + C(m,q),
\end{align*}
so that
\begin{equation*}
	\frac{d}{d\zeta} \big( U_{1,K} + V_{1,K} \big) \le C(K,m,q) - \frac{\nu-1}{\nu} U_{1,K}^{2\nu} - \frac{V_{1,K}^2}{4(m-1)} .
\end{equation*}
It readily follows from the above differential inequality that $(U_{1,K},V_{1,K})$ is bounded on $[0,\zeta_{1,K}^+)$, while~\eqref{p030} and~\eqref{p035} ensure that $dU_{1,K}/d\zeta>0$ on $[0,\zeta_{1,K}^+)$. Therefore, $\zeta_{1,K}^+=\infty$ and there exists $U_\infty>0$ such that
\begin{equation*}
	\lim_{\zeta\to\infty} U_{1,K}(\zeta) = U_\infty>0. %\label{p036}
\end{equation*}
We now argue as at the end of \textbf{Step~2} to show that
\begin{equation*}
	\omega_{1,K} := \{ \text{cluster points in }\; \mathbb{R}^2 \;\text{ of }\; (U_{1,K},V_{1,K})(\zeta) \;\text{ as }\; \zeta\to \infty \} = \{Q_2\},
\end{equation*}
from which we deduce that $l_1(K)$ connects $Q_1$ to $Q_2$. However, $Q_2$ is an unstable node as $K\in (0,K_u)$ and a contradiction. Consequently, $l_1(K)$ crosses the half-line $(0,\infty)\times \{0\}$ at a finite point.
\end{proof}

%%%%%%%%%%%%%%%%%%%%
\begin{remark}\label{rem.3}
	The first statement in Proposition~\ref{prop.Ksmall} already leads to the proof of Theorem~\ref{th.2}. However, the end of the proof is postponed to Section~\ref{sec.proofs} at the end of the paper, since more precise information on the dynamics of~\eqref{p030} is needed to prove Theorem~\ref{th.1}.
\end{remark}
%%%%%%%%%%%%%%%%%%%%

%%%%%%%%%%%%%%%%%%%%
%%%%%%%%%%%%%%%%%%%%
\subsection{Invariant regions}\label{subsec.irm}
%%%%%%%%%%%%%%%%%%%%
%%%%%%%%%%%%%%%%%%%%

Let $K>0$. Recalling the definition~\eqref{p049} of the trajectory $(U_{1,K},V_{1,K})$, we define
\begin{equation*}
	\begin{split}
		\zeta_{1,K} & := \inf\{ \zeta\in (-\infty,\zeta_{1,K}^+)\ :\ V_{1,K}(\zeta)=0 \} \in (-\infty,\infty], \\
		U_1(K) & := \lim_{\zeta\nearrow\zeta_{1,K}} U_{1,K}(\zeta) \in (0,\infty],
	\end{split} %\label{p052}
\end{equation*}
the latter being a consequence of the monotonicity of $U_{1,K}$ on $(-\infty,\zeta_{1,K})$.

We gather preliminary information on $U_1(K)$ in the next lemma.
	
%%%%%%%%%%%%%%%%%%%%
\begin{lemma}\label{lem.U1K}
	Let $K>0$.
	\begin{itemize}
		\item [(a)] If $U_1(K)<\infty$ and $\zeta_{1,K}<\infty$, then
		\begin{equation*}
			U_{1}(K)>1, \quad V_{1,K}(\zeta_{1,K})=0, \;\text{ and }\;\; \frac{dV_{1,K}}{d\zeta}(\zeta_{1,K})<0.
		\end{equation*}
		\item [(b)] If $U_1(K)<\infty$ and $\zeta_{1,K}=\infty$, then $U_{1}(K)=1$ and
		\begin{equation*}
			\lim\limits_{\zeta\to\infty} \big(U_{1,K},V_{1,K}\big)(\zeta) = Q_2.
		\end{equation*}
	\end{itemize}
\end{lemma}
%%%%%%%%%%%%%%%%%%%%

\begin{proof} (a) The definition of $\zeta_{1,K}$, along with the monotonicity of $U_{1,K}$ on $(-\infty,\zeta_{1,K})$ and~\eqref{p030} entails that $U_1(K) = U_{1,K}(\zeta_{1,K})>0$, $V_{1,K}(\zeta_{1,K}) = 0$, and
\begin{equation*}
	 0 \ge \frac{dV_{1,K}}{d\zeta}(\zeta_{1,K}) = U_1(K) \big( 1 - U_1(K)^{2\nu-1}\big).
\end{equation*}
Furthermore, if $U_1(K)=1$, then $(U_{1,K},V_{1,K})(\zeta)=Q_2$ for all $\zeta\in\mathbb{R}$, which contradicts the definition of $l_1(K)$. Consequently, $U_1(K)\ne 1$ and thus $U_1(K)>1$ with $\frac{dV_{1,K}}{d\zeta}(\zeta_{1,K})<0$, as claimed.

\medskip

\noindent (b) Since $\zeta_{1,K}=\infty$, the function $V_{1,K}$ is positive on $\mathbb{R}$ and so is $U_{1,K}$ due to its monotonicity. Then $l_1(K)\subset (0,\infty)^2$ and we argue as in \textbf{Step~3} of the proof of Proposition~\ref{prop.Ksmall} to conclude that $(U_{1,K},V_{1,K})(\zeta)$ converges to $Q_2$ as $\zeta\to\infty$.
\end{proof}

Next, since $U_{1,K}$ is increasing from $(-\infty,\zeta_{1,K})$ onto $(0,U_1(K))$, then it is invertible on that set, so that we may define $b_{1,K} := V_{1,K}\circ U_{1,K}^{-1}$. Then $b_{1,K}$ maps $(0,U_{1}(K))$ onto $V_{1,K}((-\infty,\zeta_{1,K}))$. Moreover, we can extend $b_{1,K}$ by continuity with $b_{1,K}(0)=(m-1)C_1(K)$ and $b_{1,K}(U_1(K))=0$ when $U_1(K)<\infty$, see Lemma~\ref{lem.U1K}. We next define
\begin{equation*}
	\mathcal{B}_{1,K} := \big\{ (u,v) \in (0,U_1(K))\times \real\ :\ 0 < v < b_{1,K}(u) \big\}. %\label{p053}
\end{equation*}
The next result establishes that $\mathcal{B}_{1,K}$ is an invariant region for the system~\eqref{p030} with larger values of $K$.

%%%%%%%%%%%%%%%%%%%%
\begin{proposition}\label{prop.1}
Assume that $K_2>K_1>0$ with $U_1(K_1)<\infty$ and consider $(U_0,V_0)\in\mathcal{B}_{1,K_1}$. Then
\begin{equation*}
	(U,V)(\zeta) = \Psi_{K_2}(\zeta,U_0,V_0)\in\mathcal{B}_{1,K_1}
\end{equation*}
for any $\zeta\in\big[0,\min\{\zeta_{K_2}^{+}(U_0,V_0),\zeta_{K_2}^{\circ}(U_0,V_0)\}\big)$, where
\begin{equation*}
	\zeta_{K_2}^{\circ}(U_0,V_0):=\inf\{\zeta\in(0,\zeta_{K_2}^{+}(U_0,V_0))\ :\ V(\zeta)=0\}.
\end{equation*}
Moreover, if $\zeta_{K_2}^{\circ}(U_0,V_0)<\zeta_{K_2}^{+}(U_0,V_0)$, then $V(\zeta_{K_2}^{\circ}(U_0,V_0))=0$ and $U(\zeta_{K_2}^{\circ}(U_0,V_0))\in(1,U_1(K_1)]$.
\end{proposition}
%%%%%%%%%%%%%%%%%%%%

\begin{proof}
Since $(U_0,V_0)\in(0,\infty)^2$ and $V>0$ on $[0,\zeta_{K_2}^{\circ}(U_0,V_0))$, we deduce from~\eqref{p030} that $U$ is increasing on $\big[0,\zeta_{K_2}^{\circ}(U_0,V_0)\big)$ and thus positive. Moreover, for $u\in(0,U_1(K_1))$, there is $\zeta_u\in(-\infty,\zeta_{1,K_1})$ such that $U_{1,K_1}(\zeta_u)=u$ and
$$
b_{1,K_1}(u)=b_{1,K_1}\circ U_{1,K_1}(\zeta_u)=V_{1,K_1}(\zeta_u), \quad b_{1,K_1}'(u)=\frac{dV_{1,K_1}/d\zeta}{dU_{1,K_1}/d\zeta}(\zeta_u).
$$
Consequently,
\begin{equation*}
\begin{split}
S_{2,K_2}(u,b_{1,K_1}(u))&-b_{1,K_1}'(u)S_{1,K_2}(u,b_{1,K_1}(u))\\
&=S_{2,K_2}(u,b_{1,K_1}(u)) -  \frac{dV_{1,K_1}/d\zeta}{dU_{1,K_1}/d\zeta}(\zeta_u) S_{1,K_2}(u,b_{1,K_1}(u))\\
&=\left(S_{2,K_2}-\frac{S_{2,K_1}}{S_{1,K_1}}S_{1,K_2}\right)(U_{1,K_1}(\zeta_u),V_{1,K_1}(\zeta_u))\\
&=(C_1(K_2)-C_1(K_1))V_{1,K_1}(\zeta_u)<0,
\end{split}
\end{equation*}
which shows that the direction of the flow of the system~\eqref{p030} corresponding to $K=K_2$ across the curve
\begin{equation*}
	\mathcal{C}_{1,K_1} := \big\{ (u,v)\in (0,U_1(K_1))\times\mathbb{R}\ :\ v=b_{1,K_1}(u))\big\}
\end{equation*}
points towards the interior of the region $\mathcal{B}_{1,K_1}$. We thus infer that $(U(\zeta),V(\zeta))$ cannot leave the region $\mathcal{B}_{1,K_1}$ for $0 < \zeta < \min\{\zeta_{K_2}^{+}(u_0,v_0),\zeta_{K_2}^{\circ}(u_0,v_0)\}$, as claimed.

Finally, if $\zeta_{K_2}^{\circ}(U_0,V_0)<\zeta_{K_2}^{+}(U_0,V_0)$, then $\zeta_{K_2}^{\circ}(U_0,V_0)<\infty$ and $V(\zeta_{K_2}^{\circ}(U_0,V_0))=0$. On the one hand, $(U,V)\big(\big[0,\zeta_{K_2}^\circ(U_0,V_0)\big)\big)\subset \mathcal{B}_{1,K_1}$ implies that $U\big(\zeta_{K_2}^\circ(U_0,V_0)\big)\le U_1(K_1)$. On the other hand, it follows from the definition of $\zeta_{K_2}^\circ(U_0,V_0)$ and~\eqref{p030} that
\begin{equation}
	0\ge \frac{dV}{d\zeta}(\zeta_{K_2}^{\circ}(U_0,V_0)) = U(\zeta_{K_2}^{\circ}(U_0,V_0)) (1-U(\zeta_{K_2}^{\circ}(U_0,V_0))^{2\nu-1}). \label{al01}
\end{equation}
Assume for contradiction that $\frac{dV}{d\zeta}(\zeta_{K_2}^{\circ}(U_0,V_0))=0$. Then $(U,V)\big(\zeta_{K_2}^{\circ}(U_0,V_0)\big) = Q_2$ due to~\eqref{al01} and the positivity and monotonicity of $U$ on $\big[0,\zeta_{K_2}^\circ(U_0,V_0)\big)$ and thus $(U_0,V_0) = Q_2$, which contradicts the positivity of $V_0$. Consequently, $\frac{dV}{d\zeta}(\zeta_{K_2}^{\circ}(U_0,V_0))<0$ and the positivity of $U$ on $\big[0,\zeta_{K_2}^{\circ}(U_0,V_0)\big]$ entails that $U\big(\zeta_{K_2}^{\circ}(U_0,V_0)\big) > 1$.
\end{proof}

The fact that $U_1(K)<\infty$ for $K>K_1$, provided $U_1(K_1)<\infty$, follows as an immediate consequence of the previous invariance result.

%%%%%%%%%%%%%%%%%%%%
\begin{corollary}\label{cor1}
For all $K>0$, we have $U_1(K)<\infty$ and, if $K_2>K_1$, then there holds $(U_{1,K_2},V_{1,K_2})(\zeta)\in\mathcal{B}_{1,K_1}$ for any $\zeta\in(-\infty,\zeta_{1,K_2})$, as well as $U_1(K_2)\le U_1(K_1)$. Moreover, for  $K\in(0,\infty)$, either $\zeta_{1,K}<\infty$ with $V_{1,K}(\zeta_{1,K})=0$, $\frac{dV_{1,K}}{d\zeta}(\zeta_{1,K})<0$, and $U_{1}(K)>1$, or $\zeta_{1,K}=\infty$ with $\lim\limits_{\zeta\to\infty}V_{1,K}(\zeta)=0$ and $U_1(K)=1$.
\end{corollary}
%%%%%%%%%%%%%%%%%%%%

\begin{proof}
We first infer from Proposition~\ref{prop.Ksmall} that $U_1(K)<\infty$ for $K\in (0,K_\star)$.

Consider next $K_1>0$ such that $U_1(K_1)<\infty$ and $K_2>K_1$. Since
$$
\lim\limits_{\zeta\to-\infty}(U_{1,K_2},V_{1,K_2})(\zeta)=(0,(m-1)C_1(K_2)),
$$
with $U_{1,K_2}(\zeta)>0$ for $\zeta\in(-\infty,\zeta_{1,K_2})$, and $(m-1)C_1(K_2)<(m-1)C_1(K_1)$, there is $\zeta_2\in(-\infty,\zeta_{1,K_2})$ such that $(U_{1,K_2}(\zeta),V_{1,K_2}(\zeta))\in\mathcal{B}_{1,K_1}$ for any $\zeta\in(-\infty,\zeta_2)$. Then, by Proposition~\ref{prop.1},
\begin{equation*}
	(U_{1,K_2},V_{1,K_2})(\zeta)\in\mathcal{B}_{1,K_1} \;\;\text{ for any }\;\; \zeta\in (-\infty,\zeta_{1,K_2}).
\end{equation*}
Thus, $b_{1,K_2}(u)\le b_{1,K_1}(u)$ for $u\in \big(0,\min\{U_1(K_1),U_1(K_2)\}\big)$, from which we immediately deduce that $U_1(K_2)\le U_1(K_1)$. Applying first this property with $K_1\in (0,K_\star)$ and $K_2>K_1$ and then with arbitrary $K_2>K_1>0$ provide the first two statements of Corollary~\ref{cor1}.

Finally, let $K>0$. Since we have just proved that $U_1(K)<\infty$, the last statement of Corollary~\ref{cor1} readily follows from Lemma~\ref{lem.U1K}.
\end{proof}

Similarly, for $K>0$ and the trajectory $(U_{0,K},V_{0,K})$ in $l_0(K)$ defined in~\eqref{p049}, we define
\begin{equation*}
	\begin{split}
		\zeta_{0,K} & := \sup\{ \zeta\in (\zeta_{0,K}^-,\infty)\ :\ V_{0,K}(\zeta)=0 \} \in [-\infty,\infty), \\
		U_0(K) & := \lim_{\zeta\searrow\zeta_{0,K}} U_{0,K}(\zeta) \in (0,\infty],
	\end{split} %\label{p054}
\end{equation*}
and deduce from~\eqref{p050} that
\begin{equation}
	U_{0,K}>0>V_{0,K} \;\;\text {and }\;\; \frac{dU_{0,K}}{d\zeta}<0 \;\;\text{ on }\;\; (\zeta_{0,K},\infty). \label{p070}
\end{equation}

We first collect preliminary information on $U_0(K)$.
	
%%%%%%%%%%%%%%%%%%%%
\begin{lemma}\label{lem.U0K}
Let $K>0$.
\begin{itemize}
	\item [(a)] If $U_0(K)<\infty$ and $\zeta_{0,K}>-\infty$, then
	\begin{equation*}
		U_{0}(K)>1, \quad V_{0,K}(\zeta_{0,K})=0, \;\text{ and }\;\; \frac{dV_{0,K}}{d\zeta}(\zeta_{0,K})<0.	
	\end{equation*}
	\item [(b)] If $U_0(K)<\infty$ and $\zeta_{0,K}=-\infty$, then $U_{0}(K)=1$ and
	\begin{equation*}
		\lim\limits_{\zeta\to-\infty} \big(U_{0,K},V_{0,K}\big)(\zeta) = Q_2.
	\end{equation*}
	\item [(c)] If $U_0(K)=\infty$ then $\zeta_{0,K}=-\infty$ and $V_{0,K}<0$ on $\mathbb{R}$.
\end{itemize}
\end{lemma}
%%%%%%%%%%%%%%%%%%%%

\begin{proof} (a) The proof is similar to that of part~(a) of Lemma~\ref{lem.U1K}, to which we refer.
		
\medskip
		
\noindent (b) By~\eqref{p070}, $V_{0,K}<0$ and $0<U_{0,K}<U_0(K)$ on $\mathbb{R}$ and we infer from~\eqref{p030} and Young's inequality that, for $\zeta\in (-\infty,0)$,
\begin{align*}
	\frac{dV_{0,K}}{d\zeta} - C_1(K) V_{0,K} & \le - \frac{V_{0,K}^2}{m-1} + U_0(K) - C_2 U_0(K)^\nu V_{0,K} \\
	& \le \frac{V_{0,K}^2}{2(m-1)} + \frac{C_2^2 U_0(K)^{2\nu}}{2} - \frac{V_{0,K}^2}{m-1} + U_0(K) \\
	& \le  C(m,K) := U_0(K) +  \frac{C_2^2 U_0(K)^{2\nu}}{2}.
\end{align*}
Integrating the above differential inequality over $(\zeta,0)$ gives
\begin{equation*}
	V_{0,K}(0) e^{C_1(K)\zeta} + \frac{C(m,K)}{C_1(K)} \left( e^{C_1(K)\zeta} - 1 \right) \le V_{0,K}(\zeta), \qquad \zeta\in (-\infty,0),
\end{equation*}
from which we deduce that $V_{0,K}$ is bounded in $(-\infty,0)$. We then argue as in the end of \textbf{Step~2} in the proof of Proposition~\ref{prop.Ksmall} to conclude that $(U_{0,K},V_{0,K})(\zeta)$ converges to $Q_2$ as $\zeta\to-\infty$.

\medskip

\noindent (c) It follows from~\eqref{p030} and~\eqref{p070} that, for $\zeta\in (\zeta_{0,K},0)$,
\begin{equation*}
	\ln{U_{0,K}(0)} - \ln{U_{0,K}(\zeta)} = \frac{1}{\nu} \int_\zeta^0 V_{0,K}(\zeta_*)\, d\zeta_* = - \frac{1}{\nu} \int_\zeta^0 \big|V_{0,K}(\zeta_*)\big|\, d\zeta_*.
\end{equation*}
Letting $\zeta\searrow\zeta_{0,K}$ gives
\begin{equation*}
	\int_{\zeta_{0,K}}^0 \big|V_{0,K}(\zeta_*)\big|\, d\zeta_* = \infty,
\end{equation*}
hence $\zeta_{0,K}=-\infty$ owing to the continuity of $V_{0,K}$. The negativity of $V_{0,K}$ on $\mathbb{R}$ then follows from the definition of $\zeta_{0,K}$.
\end{proof}
	
Using again~\eqref{p070}, we note that $U_{0,K}$ is decreasing from $(\zeta_{0,K},\infty)$ onto $(0,U_0(K))$ and is thus invertible on that set. We then define $b_{0,K} := V_{0,K}\circ U_{0,K}^{-1}$, which maps $(0,U_{0,K})$ on $V_{0,K}((\zeta_{0,K},\infty))$. Moreover, we extend $b_{0,K}$ by continuity with $b_{0,K}(0)=b_{0,K}(U_0(K))=0$ when $U_0(K)<\infty$, see Lemma~\ref{lem.U0K}. We next define
\begin{equation*}
	\begin{split}
		\mathcal{B}_{0,K}^- & := \big\{ (u,v) \in (0,U_0(K))\times \real\ :\ v < b_{0,K}(u) \big\}, \\
		\mathcal{B}_{0,K}^+ & := \big\{ (u,v) \in (0,U_0(K))\times \real\ :\ 0> v > b_{0,K}(u) \big\},
	\end{split} %\label{p055}
\end{equation*}
and establish the invariance of the regions $\mathcal{B}_{0,K}^{-}$ and $\mathcal{B}_{0,K}^{+}$.

%%%%%%%%%%%%%%%%%%%%
\begin{proposition}\label{prop.0}
(a) Let $K_2>K_1>0$ and consider $(U_0,V_0) \in \mathcal{B}_{0,K_1}^{-}$. Then
\begin{equation*}
	(U,V)(\zeta) = \Psi_{K_2}(\zeta,U_0,V_0) \in\mathcal{B}_{0,K_1}^{-}
\end{equation*}
for any $\zeta\in \big( \max\{\zeta_{K_2}^{-}(U_0,V_0), \zeta_{K_2}^{K_1}(U_0,V_0)\},0\big)$, where
\begin{equation*}
\begin{split}
	\zeta_{K_2}^{K_1}(U_0,V_0) & := - \infty \;\text{ if }\; U_0(K_1)=\infty, \\ \zeta_{K_2}^{K_1}(U_0,V_0) & := \sup\{\zeta\in(\zeta_{K_2}^{-}(U_0,V_0),0)\ :\  U(\zeta)=U_0(K_1)\} \;\text{ otherwise }.
\end{split}
\end{equation*}

(b) Let $K_1>K_2>0$ and consider $(U_0,V_0) \in \mathcal{B}_{0,K_1}^{+}$. Then
\begin{equation*}
	(U,V)(\zeta) = \Psi_{K_2}(\zeta,U_0,V_0) \in\mathcal{B}_{0,K_1}^{+}
\end{equation*}
for any $\zeta\in \big(\zeta_{K_2}^{\circ,-}(U_0,V_0),0\big)$,
where
\begin{equation*}
\zeta_{K_2}^{\circ,-}(U_0,V_0):=\sup\{\zeta\in(\zeta_{K_2}^{-}(U_0,V_0),0): V(\zeta)=0\}.
\end{equation*}
\end{proposition}
%%%%%%%%%%%%%%%%%%%%

\begin{proof}
The proof is completely similar to that of Proposition~\ref{prop.1}, noticing that, for any $u\in(0,U_0(K_1))$ we have
$$
S_{2,K_2}(u,b_{0,K_1}(u))-b_{0,K_1}'(u)S_{1,K_2}(u,b_{0,K_1}(u))=(C_1(K_2)-C_1(K_1))b_{0,K_1}(u),
$$
and the conclusion follows in both cases from the monotonicity of $K\mapsto C_1(K)$, the negativity of $b_{0,K_1}$, and the fact that $U$ is decreasing on $\big( \zeta_{K_2}^{\circ,-}(U_0,V_0),0\big)$ due to the non-positivity of $V$ on that interval. We omit the details.
\end{proof}

Particularizing the previous result to the trajectory $(U_{0,K},V_{0,K})$, we obtain the following consequence.

%%%%%%%%%%%%%%%%%%%%
\begin{corollary}\label{cor0}
Let $K_2>K_1$. Then there is $\zeta_{0,K_2}^{K_1}\in [-\infty,\infty)$ such that $(U_{0,K_2},V_{0,K_2})(\zeta)\in\mathcal{B}_{0,K_1}^{-}$ for any $\zeta\in \big(\zeta_{0,K_2}^{K_1},\infty\big)$. In particular, $U_0(K_2)=\infty$ if $U_0(K_1)=\infty$ and $U_0(K_2)\geq U_0(K_1)$ otherwise.
\end{corollary}
%%%%%%%%%%%%%%%%%%%%

\begin{proof}
Since
$$
\lim\limits_{\zeta\to\infty}\frac{V_{0,K_2}(\zeta)}{U_{0,K_2}(\zeta)}=-\frac{1}{C_1(K_2)}<-\frac{1}{C_1(K_1)}
=\lim\limits_{\zeta\to\infty}\frac{V_{0,K_1}(\zeta)}{U_{0,K_1}(\zeta)},
$$
we have
\begin{equation*}
	\lim\limits_{u\to 0} \frac{b_{0,K_2}(u)}{u} = - \frac{1}{C_1(K_2)} < - \frac{1}{C_1(K_1)} = \lim\limits_{u\to 0} \frac{b_{0,K_1}(u)}{u}.
\end{equation*}
Consequently, $b_{0,K_2}<b_{0,K_1}$ in a right neighborhood of $u=0$, and we readily infer that $(U_{0,K_2},V_{0,K_2})(\zeta)\in\mathcal{B}_{0,K_1}^{-}$ for $\zeta\in(\zeta_0,\infty)$ for some $\zeta_0$ sufficiently large. The first statement of Corollary~\ref{cor0} then follows from part~(a) of Proposition~\ref{prop.0} with $\zeta_{0,K_2}^{K_1} := \zeta_{K_2}^{K_1}\big(U_{0,K_2}(\zeta_0),V_{0,K_2}(\zeta_0)\big)$.

Assume next for contradiction that $U_0(K_2)<U_0(K_1)$, which implies in particular that $U_0(K_2)<\infty$. Then
\begin{equation}
	V_{0,K_2}(\zeta) < b_{0,K_1}(\zeta) < 0, \qquad \zeta\in \big(\zeta_{0,K_2}^{K_1},\infty\big), \label{p081}
\end{equation}
from which we deduce that $\zeta_{0,K_2}\le \zeta_{0,K_2}^{K_1}$. The monotonicity of $U_{0,K_2}$ on $(\zeta_{0,K_2},\infty)$ then guarantees that $U_{0,K_2}(\zeta)<U_0(K_2)<U_0(K_1)$ for $\zeta\in \big(\zeta_{0,K_2}^{K_1},\infty)$ and thus $\zeta_{0,K_2}^{K_1} =\zeta_{0,K_2}$. Now,either $\zeta_{0,K_2} > -\infty$ and we infer from~\eqref{p081} and Lemma~\ref{lem.U0K} that $0<b_{0,K_1}(U_0(K_2))<0$, and a contradiction. Or $\zeta_{0,K_2}=-\infty$ and it follows from~\eqref{p081} and Lemma~\ref{lem.U0K} that $U_0(K_1)>U_0(K_2)=1$ and $0\le b_{0,K_1}(1)<0$, and again a contradiction. Therefore, $U_0(K_2)\ge U_0(K_1)$ as claimed.
\end{proof}

Further consequences of Proposition~\ref{prop.0} and Corollary~\ref{cor0} are collected in the next result.

%%%%%%%%%%%%%%%%%%%%
\begin{corollary}\label{cor00}
Let $K>0$.
\begin{itemize}
	\item [(a)] if $U_0(K)\in (1,\infty)$, then $\zeta_{0,K}>-\infty$ with $V_{0,K}(\zeta_{0,K})=0$ and $\frac{dV_{0,K}}{d\zeta}(\zeta_{0,K})<0$;
	\item [(b)] if $U_0(K)=1$, then $\zeta_{0,K}=-\infty$, $V_{0,K}<0$ on $\mathbb{R}$, and $l_0(K)$ is a complete trajectory connecting $Q_2$ to $Q_0$. In particular, $l_0(K) \subset W_u(Q_2)$ and $K\in (0,K_u]$, where $K_u$ is defined in Lemma~\ref{lem.Q2}.
\end{itemize}
\end{corollary}
%%%%%%%%%%%%%%%%%%%%

\begin{proof}
The first statement of Corollary~\ref{cor00} follows from Lemma~\ref{lem.U0K}.
	
Let us next consider the case where $K>0$ is such that $U_0(K)=1$. According to Lemma~\ref{lem.U0K}, we have $\zeta_{0,K}=-\infty$ in that case and $(U_{0,K},V_{0,K})(\zeta)$ converges to $Q_2$ as $\zeta\to-\infty$. Then $l_0(K)$ is a heteroclinic orbit of~\eqref{p030} connecting $Q_2$ and $Q_0$ which thus lies in $W_u(Q_2)$. Such a situation only occurs when $K\in (0,K_u]$ according to Lemma~\ref{lem.Q2} and the definition of a focus.
\end{proof}

We need one more preparatory result related to the orbit $l_0(K)$ for some $K>0$; more precisely, we have:

%%%%%%%%%%%%%%%%%%%%
\begin{lemma}\label{lem.amplitude}
If $K>0$ is such that $U_0(K)\in (1,\infty)$, then $U_{0,K}(\zeta)\le U_0(K)$ for all $\zeta\in(\zeta_{0,K}^-,\infty)$. In other words, $U_0(K)$ is the highest value of the coordinate $u$ attained along the trajectory $l_0(K)$.
\end{lemma}
%%%%%%%%%%%%%%%%%%%%

\begin{proof}
Let $K>0$ be such that $U_0(K)\in(1,\infty)$. First, recalling~\eqref{p070}, $U_{0,K}$ is decreasing on $(\zeta_{0,K},\infty)$, whence $U_{0,K}(\zeta)<U_0(K)$ for $\zeta\in (\zeta_{0,K},\infty)$. It remains to show the same result for $\zeta\in(\zeta_{0,K}^-,\zeta_{0,K})$. According to Corollary~\ref{cor00}~(a), $\zeta_{0,K}>\zeta_{0,K}^-$ and there is $\delta>0$ such that $V_{0,K}(\zeta)>0$ for $\zeta\in (\zeta_{0,K}-\delta,\zeta_{0,K})$. This property, along with the first equation of~\eqref{p030}, implies that $\frac{dU_{0,K}}{d\zeta}(\zeta)>0$ and thus $U_{0,K}(\zeta)<U_0(K)$ in the same left neighborhood $(\zeta_{0,K}-\delta,\zeta_{0,K})$. Tracking back the orbit $(U_{0,K},V_{0,K})$, we have two possibilities:

$\bullet$ either $V_{0,K}(\zeta)>0$ for any $\zeta\in(\zeta_{0,K}^-,\zeta_{0,K})$, so that $\frac{dU_{0,K}}{d\zeta}>0$ on $(\zeta_{0,K}^-,\zeta_{0,K})$ and thus $U_{0,K}(\zeta)<U_0(K)$ for any $\zeta\in(\zeta_{0,K}^-,\zeta_{0,K})$, completing the proof.

$\bullet$ or there is a last zero of $V_{0,K}$ prior to $\zeta_{0,K}$; that is,
\begin{equation*}
	z_1 := \sup\{\zeta\in(\zeta_{0,K}^-,\zeta_{0,K}): V_{0,K}(\zeta)=0\} \in \big(\zeta_{0,K}^-,\zeta_{0,K}\big).
\end{equation*}
Then $V_{0,K}(\zeta)>0$ and thus $U_{0,K}(\zeta)<U_0(K)$ for $\zeta\in(z_1,\zeta_{0,K})$. Moreover, $V_{0,K}(z_1)=0$ with $\frac{dV_{0,K}}{d\zeta}(z_1)\ge 0$ and the finiteness of $z_1$ prevents $(U_{0,K}(z_1),0)$ from being a critical point, whence $U_{0,K}(z_1)\not\in\{0,1\}$. Since
\begin{equation*}
	U_{0,K}(z_1) \left( 1 - U_{0,K}(z_1)^{2\nu-1} \right) = \frac{dV_{0,K}}{d\zeta}(z_1)\ge 0,
\end{equation*}
we conclude that $U_{0,K}(z_1)\in (0,1)$ and $\frac{dV_{0,K}}{d\zeta}(z_1)<0$. Consequently, there is $\delta_1>0$ such that $U_{0,K}\in (0,1)$ and $V_{0,K}<0$ on $(z_1-\delta_1,z_1)$. In particular, $(U_{0,K}(\zeta),V_{0,K}(\zeta))\in \mathcal{B}_{0,K}^+$ for $\zeta\in (z_1-\delta_1,z_1)$ and we deduce from Proposition~\ref{prop.0}~(b) that
\begin{equation*}
	(U_{0,K}(\zeta),V_{0,K}(\zeta))\in \mathcal{B}_{0,K}^+, \qquad \zeta\in (z_2,z_1),
\end{equation*}
with
\begin{equation*}
	z_2 := \sup\{\zeta\in(\zeta_{0,K}^-,z_1): V_{0,K}(\zeta)=0\}.
\end{equation*}
Since the curve
\begin{equation}
	\mathcal{C}_{0,K} := \left\{(u,v)\in (0,U_0(K))\times \real,\:\ v=b_{0,K}(u)\right\} \;\text{ is a separatrix for }\; \Psi_{K}, \label{y005}
\end{equation}
we conclude that $b_{0,K}(U_{0,K}(\zeta))< V_{0,K}(\zeta)<0$ for $\zeta\in (z_2,z_1)$ and thus $U_{0,K}(\zeta)<U_0(K)$ for $\zeta\in (z_2,z_1)$. Now,

\hspace{1cm} $\bullet$ either  $z_2=\zeta_{0,K}^-$ and the relative compactness of $\mathcal{B}_{0,K}^+$ entails that $\zeta_{0,K}^-=-\infty$, as well as the relative compactness of $(U_{0,K},V_{0,K})((-\infty,z_1))$ and the monotonicity of $U_{0,K}$ on $(-\infty,z_1)$. We then argue as at the end of the proof of \textbf{Step~2} in Proposition~\ref{prop.Ksmall} to conclude that $(U_{0,K},V_{0,K})(\zeta)$ converges to $Q_2$ as $\zeta\to -\infty$ with $U_{0,K}(z_1)<U_{0,K}(\zeta)<1<U_0(K)$ for $z\in (-\infty,z_1)$.

\hspace{1cm} $\bullet$ Or $z_2>\zeta_{0,K}^-$ and we argue as above to obtain that
\begin{equation*}
	U_{0,K}(z_2)\in (1,U_0(K)), \quad V_{0,K}(z_2)=0, \quad \frac{dV_{0,K}}{d\zeta}(z_2)<0.
\end{equation*}
At this point, either $V_{0,K}>0$ on $(\zeta_{0,K}^-,z_2)$ or
\begin{equation*}
	z_3 := \sup\{\zeta\in(\zeta_{0,K}^-,z_2): V_{0,K}(\zeta)=0\} \in (\zeta_{0,K}^-,z_2).
\end{equation*}
We then proceed as above to deduce that $U_{0,K}<U_0(K)$ on $(\zeta_{0,K}^-,z_2)$ in the first case and on $(z_3,z_2)$ in the second case. Iterating this argument completes the proof.
\end{proof}

The next result identifies the range of the mapping $K\mapsto U_0(K)$.

%%%%%%%%%%%%%%%%%%%%
\begin{lemma}\label{lem.U0}
There exist $0<K_0\leq K_{\infty}<\infty$ such that $U_0(K)=1$ for $K\in(0,K_0)$ and $U_0(K)=\infty$ for $K>K_{\infty}$. In addition, $K_0\le K_u$.
\end{lemma}
%%%%%%%%%%%%%%%%%%%%

\begin{proof} \textbf{Step 1}. Set
\begin{equation*}%\label{def.K0}
	K_0:=\sup\{K>0: U_0(K)=1\}.
\end{equation*}
Proposition~\ref{prop.Ksmall} guarantees that $K_0\ge K_\star>0$, while Corollary~\ref{cor0} and the positivity of $K_\star$ imply that $U_0(K)=1$ for $K\in (0,K_0)$. Moreover, it follows from Corollary~\ref{cor00}~(b) that $K_0\le K_u$.

\medskip

\textbf{Step~2.} Introduce
$$
K_{\infty}:=\sup\{K>0:U_0(K)<\infty\}\in (0,\infty].
$$
On the one hand, $K_{\infty}\geq K_0>0$. On the other hand, it follows from \cite{IL24} that there is a unique pair of exponents  $(\alpha^*,\beta^*)\in (0,\infty)^2$, $\alpha^*=2\beta^*/(m-1)$ such that~\eqref{ODE} has a compactly supported solution $f^*$ with support $[0,\xi_0^*]$ which is decreasing on $[0,\xi_0^*]$ and satisfies
\begin{equation*}
	f^*(0)=1, \quad (f^*)'(0)=0, \quad f^*(\xi_0^*) = 0, \quad [(f^*)^m]'(\xi_0^*)=0.
\end{equation*}
Moreover, by \cite[Proposition~2.8~(c)]{IL24},
\begin{equation}
\begin{split}
	f^*(\xi) & \sim \left( \frac{1-q}{\beta^*} \right)^{1/(1-q)} (\xi_0^*)^{(\sigma-1)/(1-q)} \big( \xi_0^* - \xi \big)^{1/(1-q)}, \\
	[(f^*)^m]'(\xi) & \sim - \frac{m}{1-q}  \left( \frac{1-q}{\beta^*} \right)^{m/(1-q)} (\xi_0^*)^{m(\sigma-1)/(1-q)} \big( \xi_0^* - \xi \big)^{(m+q-1)/(1-q)},
\end{split}
\label{y002}
\end{equation}
as $\xi\nearrow\xi_0^*$. As in Section~\ref{subsec.aaf}, we define
\begin{equation*}
	\begin{split}
		\mathcal{X}_{f^*}(\xi) & := \frac{2m}{\alpha^*} \xi^{-2} (f^*)^{m-1}(\xi), \ \mathcal{Y}_{f^*}(\xi) := \frac{2m}{\alpha^*}\xi^{-1} (f^*)^{m-2}(\xi) (f^*)'(\xi),  \\
		\mathcal{Z}_{f^*}(\xi) & := \mathcal{X}_{f^*}(\xi)^{1/\nu}=\left[\frac{2m}{\alpha^*}\right]^{1/\nu} \xi^{-2/\nu} (f^*)^{\mu}(\xi),
	\end{split}
	\end{equation*}
for $\xi\in (0,\xi_0^*)$ and
\begin{equation*}
	\eta_{f^*}(\xi) := \frac{\alpha^*}{2m} \int_0^\xi \frac{s}{(f^*)^{m-1}(s)}\,ds, \qquad \xi\in (0,\xi_0^*).
\end{equation*}
Since $m+q>2$, $\lim_{\xi\to\xi_0^*} \eta_{f^*}(\xi)=\infty$ and the functions $\big(Y^*,Z^*\big) := \big(\mathcal{Y}_{f^*}\circ \eta_{f^*}, \mathcal{Z}_{f^*}\circ \eta_{f^*}\big)$ solve~\eqref{p018a} on $(0,\infty)$ with
\begin{equation*}
	K^*:= \frac{1}{m} \left[ \frac{2m}{\alpha^*} \right]^{(m-q)/(m-1)}.
\end{equation*}
Moreover, the properties of $f^*$ and~\eqref{y002} guarantee that
\begin{equation}
	\lim_{\eta\to 0} \big(Y^*,Z^*\big)(\eta) = \left( - \frac{2}{N} , \infty \right), \quad \lim_{\eta\to\infty} \big(Y^*,Z^*\big) = (0,0), \label{y003}
\end{equation}
\begin{equation}
	Y^* < 0 \;\;\text{ and }\;\; Z^*> 0 \;\;\text{ on }\;\; (0,\infty). \label{y004}
\end{equation}
Owing to~\eqref{y003}, the local behavior of~\eqref{p018a} described in Lemma~\ref{lem.P0} implies that 
$$
\big(Y^*,Z^*\big)((0,\infty))\subset W_c(P_0),
$$ 
whence $\big(Y^*,Z^*\big)((0,\infty))=l_0(K^*) $ by Lemma~\ref{lem.P0}. Introducing
\begin{equation*}
	U^*(\zeta) := \frac{1}{Z_{K^*}} Z^*\left( \frac{\zeta}{\lambda_{K^*}} \right), \quad V^*(\zeta) := \frac{1}{\lambda_{K^*}} \left[(m-1) Y^*\left(\frac{\zeta}{\lambda_K}\right) - 2 (Z^*)^{\nu}\left(\frac{\zeta}{\lambda_K}\right) \right]
\end{equation*}
for $\zeta\in (0,\infty)$ as in~\eqref{p029}, which is nothing but another parametrization of $l_0(K^*)$, we readily deduce from~\eqref{y004} that $\big(U^*,V^*\big)(\zeta)\in (0,\infty)\times (-\infty,0)$ for all $\zeta\in (0,\infty)$. Therefore, $l_0(K^*)\subset (0,\infty)\times (-\infty,0)$ and, recalling~\eqref{p049}, we have shown that $\zeta_{0,K^*}^-=-\infty$. We next deduce from~\eqref{y003} and the definition of $\big(U^*,V^*\big)$ that
\begin{equation*}
	\lim_{\zeta\to -\infty} \big( U_{0,K^*} , V_{0,K^*} \big)(\zeta) = (\infty,-\infty).
\end{equation*}
Hence, $U_0(K^*)=\infty$ and therefore $K_{\infty}<K^*$. Moreover, it follows from Corollary~\ref{cor0} that, if $U_0(K_1)=\infty$ for some $K_1>0$, then we have $U_0(K)=\infty$ for any $K>K_1$. Therefore, $U_0(K)=\infty$ for any $K>K_\infty$.
\end{proof}

More precise information on $U_0(K_0)$ and $U_0(K_{\infty})$ will be given in Section~\ref{subsec.cd} after proving the continuous dependence.

%%%%%%%%%%%%%%%%%%%%
%%%%%%%%%%%%%%%%%%%%
\subsection{Strict monotonicity of $U_0$ and $U_1$}\label{subsec.monot}
%%%%%%%%%%%%%%%%%%%%
%%%%%%%%%%%%%%%%%%%%

The invariant regions constructed in the previous section allow us to conclude that, for any $K_2>K_1>0$, we have
\begin{equation}
	b_{1,K_2}(u)<b_{1,K_1}(u), \quad 0\leq u<U_1(K_2)\le U_1(K_1) \label{y006}
\end{equation}
and
\begin{equation}
	b_{0,K_2}(u)<b_{0,K_1}(u)<0, \quad 0\leq u<U_0(K_1)\le U_0(K_2), \label{y007}
\end{equation}
see Corollary~\ref{cor1} and Corollary~\ref{cor0}, respectively. In particular, we have only shown up to now that $U_1(K_2)\leq U_1(K_1)$ and $U_0(K_1)\leq U_0(K_2)$. However, in order to proceed further, we need strict inequalities. This is the goal of the next result.

%%%%%%%%%%%%%%%%%%%%
\begin{proposition}\label{prop.monot}
Let $0<K_1<K_2$.
\begin{itemize}
	\item [(a)] If $U_1(K_2)>1$, then $U_1(K_1)>U_1(K_2)$.
	\item [(b)] If $U_0(K_1)\in(1,\infty)$, then $U_0(K_1)<U_0(K_2)$.
\end{itemize}
\end{proposition}
%%%%%%%%%%%%%%%%%%%%

\begin{proof}
We recall that, for a general $K>0$, the function $b_{1,K}=V_{1,K}\circ U_{1,K}^{-1}$ is defined on $[0,U_1(K)]$ and the inverse function theorem gives
\begin{equation}\label{p056}
b_{1,K}'(u)=\frac{\nu C_1(K)}{u}-\frac{\nu b_{1,K}(u)}{(m-1)u}+\frac{\nu(1-u^{2\nu-1})}{b_{1,K}(u)}-\nu C_2 u^{\nu-1}, \quad u\in \big(0,U_1(K)\big).
\end{equation}
Consequently, for $K_2>K_1>0$, the function $B_1:=b_{1,K_2}-b_{1,K_1}$ is well defined on $[0,U_1(K_2)]$ and we easily infer from~\eqref{p056} that
\begin{equation}\label{p057}
B_1'(u)+\omega_1(u) B_1(u)=-\frac{\delta}{u}, \quad 0<u<U_1(K_2),
\end{equation}
where
\begin{equation*}%\label{p059}
\omega_1(u) := \nu \left[\frac{1}{(m-1)u} + \frac{1-u^{2\nu-1}}{b_{1,K_1}(u)b_{1,K_2}(u)}\right], \quad \delta:=\nu(C_1(K_1)-C_1(K_2))>0.
\end{equation*}
Since $1<U_1(K_2)\leq U_1(K_1)$, we find that $\omega_1(u)\leq\nu/(m-1)u$ for any $u\in[1,U_1(K_2))$ and we infer from~\eqref{p057} and the non-positivity~\eqref{y006} of $B_1$ on $\big(0,U_1(K_2)\big)$ that
$$
B_1'(u)+\frac{\nu}{(m-1)u}B_1(u)\leq-\frac{\delta}{u}, \quad 1\leq u<U_1(K_2),
$$
or equivalently
$$
\frac{d}{du}(B_1(u)u^{\nu/(m-1)})=u^{\nu/(m-1)}\left[B_1'(u)+\frac{\nu}{(m-1)u}B_1(u)\right]\leq-\delta u^{\nu/(m-1)-1}.
$$
Integrating on $[1,u]$ with $1<u<U_1(K_2)$, we further obtain
\begin{equation}
B_1(u)u^{\nu/(m-1)} \leq B_1(1) - \frac{\delta(m-1)}{\nu}\left(u^{\nu/(m-1)}-1\right) < B_1(1) < 0 \label{p058}
\end{equation}
for $u\in[1,U_1(K_2))$. Letting $u\to U_1(K_2)$ in the estimate~\eqref{p058}, we conclude that $B_1(U_1(K_2))<0$; that is,
\begin{equation*}
	b_{1,K_1}(U_1(K_1)) = 0 = b_{1,K_2}(U_1(K_2))<b_{1,K_1}(U_1(K_2)),
\end{equation*}
and thus $U_1(K_1)>U_1(K_2)$, as claimed.

We prove now the analogous strict monotonicity result for $K\mapsto U_0(K)$. Letting, as above, $B_0 := b_{0,K_2}-b_{0,K_1}$, which is well defined for $u\in[0,U_0(K_1)]$, we obtain in a similar way the ordinary differential equation~\eqref{p057} with $B_1$ replaced by $B_0$ and $\omega_1$ replaced by
\begin{equation}
	\omega_0(u) := \nu \left[\frac{1}{(m-1)u} + \frac{1-u^{2\nu-1}}{b_{0,K_1}(u)b_{0,K_2}(u)}\right], \qquad u\in \big(0,U_0(K_1)\big). \label{p059.5}
\end{equation}
Since $U_0(K_1)\in(1,\infty)$, we also deduce that $\omega_0(u)\leq\nu/(m-1)u$ for any $u\in[1,U_0(K_1))$. Owing to the non-positivity~\eqref{y007} of $B_0$, we may argue exactly as in the previous proof to establish that $B_0(U_0(K_1))<0$; that is,
\begin{equation*}
	b_{0,K_2}(U_0(K_1))<b_{0,K_1}(U_0(K_1))= 0= b_{0,K_2}(U_0(K_2)),
\end{equation*}
hence $U_0(K_2)>U_0(K_1)$ and the proof is complete.
\end{proof}

%%%%%%%%%%%%%%%%%%%%
%%%%%%%%%%%%%%%%%%%%
\subsection{Continuous dependence}\label{subsec.cd}
%%%%%%%%%%%%%%%%%%%%
%%%%%%%%%%%%%%%%%%%%

The last piece in the proof of Theorem~\ref{th.1} is the continuity of the mappings $K\mapsto U_0(K)$ and $K\mapsto U_1(K)$. This will be done in two steps: at first, we establish an ``interior continuity" of both; that is, assuming $U_1(K)>1$ and $U_0(K)\in(1,\infty)$. In a second step, we prove continuity only of $U_0$ as $K\to K_0$ and as $K\to K_{\infty}$, in a sense that will be made precise in this section.

%%%%%%%%%%%%%%%%%%%%
\begin{proposition}\label{prop.cd}
Let $K_1>0$.
\begin{itemize}
	\item [(a)] If $U_0(K_1)\in(1,\infty)$, then
	\begin{equation*}
		\lim\limits_{K\to K_1} U_0(K)=U_0(K_1).
	\end{equation*}
	\item [(b)] If $U_1(K_1)>1$, then
	\begin{equation*}
		\lim\limits_{K\to K_1}U_1(K)=U_1(K_1).
	\end{equation*}
	\end{itemize}
\end{proposition}
%%%%%%%%%%%%%%%%%%%%

\begin{proof}
(a) Let $K_1>0$ such that $U_0(K_1)\in(1,\infty)$. Then $\zeta_{0,K_1}>-\infty$ by Corollary~\ref{cor00}~(a) and $V_{0,K_1}(\zeta_{0,K_1})=0>\frac{dV_{0,K_1}}{d\zeta}(\zeta_{0,K_1})$. Letting $K_2>0$ and $B_0=b_{0,K_2}-b_{0,K_1}$, we obtain exactly as in the proof of Proposition~\ref{prop.monot} that
\begin{equation}\label{p060}
B_0'(u) + \omega_0(u) B_0(u) = -\frac{\nu(C_1(K_1)-C_1(K_2))}{u}, \quad u\in[0,\min\{U_0(K_1),U_0(K_2)\}),
\end{equation}
where $\omega_0$ is defined in~\eqref{p059.5}. Therefore, multiplying~\eqref{p060} by $\mathrm{sign}(B_0(u))$ and restricting to $u\in(0,1)$, we find
$$
|B_0|'(u)+\frac{\nu |B_0(u)|}{(m-1)u}\leq|B_0|'(u)+\omega_0(u)|B_0(u)| \leq \frac{\nu}{u}|C_1(K_1)-C_1(K_2)|.
$$
Equivalently,
$$
\frac{d}{du}\left[|B_0|(u)u^{\nu/(m-1)}\right] \leq \nu |C_1(K_1)-C_1(K_2)| u^{\nu/(m-1)-1}, \qquad u\in (0,1).
$$
Since $\nu>0$, by integrating over $[0,u]$ for $u\in (0,1)$ and taking the supremum with respect to $u\in [0,1]$ after dividing by $u^{\nu/(m-1)}$, we are left with
\begin{equation}\label{p061}
\sup\limits_{u\in[0,1]}|(b_{0,K_1}-b_{0,K_2})(u)|\leq(m-1)|C_1(K_1)-C_1(K_2)|.
\end{equation}
In particular, introducing $\varrho_{K_i}:=U_{0,K_i}^{-1}(1/2)$, $i=1,2$, we note that
\begin{equation*}
	\big|V_{0,K_2}(\varrho_{K_2})-V_{0,K_1}(\varrho_{K_1})\big|=\big|(b_{0,K_2}-b_{0,K_1})(1/2)\big|\leq(m-1)|C_1(K_2)-C_1(K_1)|. %\label{p061.5}
\end{equation*}
We next reparametrize the curve $l_0(K_2)$ with a shift in the independent variable by setting
\begin{equation*}
	(\bar{U}_{0,K_2},\bar{V}_{0,K_2})(\zeta):=(U_{0,K_2},V_{0,K_2})(\zeta+\varrho_{K_2}-\varrho_{K_1}), \quad \zeta>\zeta^{-}_{0,K_2}+\varrho_{K_1}-\varrho_{K_2},
\end{equation*}
which is also a solution to~\eqref{p030}. Observe that
\begin{subequations}\label{y009}
\begin{equation}	
	|\bar{U}_{0,K_2}(\varrho_{K_1})-U_{0,K_1}(\varrho_{K_1})| = |U_{0,K_2}(\varrho_{K_2})-U_{0,K_1}(\varrho_{K_1})| = 0, \label{y009a}
\end{equation}
and
\begin{equation}
	\begin{split}
	|\bar{V}_{0,K_2}(\varrho_{K_1})-V_{0,K_1}(\varrho_{0,K_1})| & = |V_{0,K_2}(\varrho_{K_2})-V_{0,K_1}(\varrho_{K_1})| \\
	& \leq(m-1)|C_1(K_2)-C_1(K_1)|.
	\end{split} \label{y009b}
\end{equation}
\end{subequations}
Now, let $\delta\in(0,1)$ be such that $\zeta^{-}_{0,K_1}<\zeta_{0,K_1}-\delta$. Since $b_{0,K_1}<0$ on $[0,1/2]$, $V_{0,K_1}<0$ on $\big(\zeta_{0,K_1}+\delta,\varrho_{K_1}\big)$ and $V_{0,K_1}(\zeta_{0,K_1}-\delta)>0$, see Corollary~\ref{cor00}~(a), we can pick $\theta\in(0,\delta)$ such that
\begin{subequations}\label{p062}
\begin{equation}\label{p062a}
b_{0,K_1}(u)\leq-2\theta, \quad u\in\left[0,\frac{1}{2}\right],
\end{equation}
\begin{equation}\label{p062b}
V_{0,K_1}(\zeta)\leq-2\theta, \quad \zeta\in[\zeta_{0,K_1}+\delta,\varrho_{K_1}],
\end{equation}
\begin{equation}\label{p062c}
V_{0,K_1}(\zeta_{0,K_1}-\delta)\geq2\theta.
\end{equation}
\end{subequations}
We then infer from~\eqref{y009} and the continuous dependence of the solutions to the system~\eqref{p030} with respect to the initial condition and $K$ that there is $\epsilon_1>0$ such that, for any $K_2\in(K_1-\epsilon_1,K_1+\epsilon_1)$,
we have
\begin{subequations}\label{p063}
\begin{equation}\label{p063b}
b_{0,K_2}(u)\leq-\theta, \quad u\in\left[0,\frac{1}{2}\right],
\end{equation}
\begin{equation}\label{p063c}
\sup\limits_{\zeta\in[\zeta_{0,K_1}-\delta,\varrho_{K_1}]} \big[ |(\bar{U}_{0,K_2} -U_{0,K_1})(\zeta)| + |(\bar{V}_{0,K_2} -V_{0,K_1})(\zeta)|\big]\leq\theta.
\end{equation}
\end{subequations}
Furthermore, we can also pick $\epsilon_0\in(0,\epsilon_1)$ such that
\begin{equation*}%\label{p064}
(m-1)|C_1(K_2)-C_1(K_1)|\leq\epsilon_1, \quad K_2\in (K_1-\epsilon_0,K_1+\epsilon_0).
\end{equation*}
Moreover, for $\zeta\in[\varrho_{K_1},\infty)$, we deduce from~\eqref{p063b} that
\begin{equation*}
\begin{split}
\bar{V}_{0,K_2}(\zeta)&=V_{0,K_2}(\zeta-\varrho_{K_1}+\varrho_{K_2})\\
&=b_{0,K_2}\big(U_{0,K_2}(\zeta-\varrho_{K_1}+\varrho_{K_2})\big)\leq-\theta,
\end{split}
\end{equation*}
the latter holding true since $\zeta-\varrho_{K_1}+\varrho_{K_2}\ge\varrho_{K_2}$ and thus the monotonicity of $U_{0,K_2}$ gives $U_{0,K_2}(\zeta-\varrho_{K_1}+\varrho_{K_2})\in[0,1/2]$. Gathering the previous estimate with~\eqref{p062} and~\eqref{p063c}, we are left with
$$
\bar{V}_{0,K_2}(\zeta)\leq-\theta, \quad \zeta\in[\zeta_{0,K_1}+\delta,\infty),
$$
while $\bar{V}_{0,K_2}(\zeta_{0,K_1}-\delta)\geq\theta$. Therefore, the largest zero of $\bar{V}_{0,K_2}$ is finite and lies in the interval $(\zeta_{0,K_1}-\delta,\zeta_{0,K_1}+\delta)$ and thus $\zeta_{0,K_2}>-\infty$ with
$$
\zeta_{0,K_2}-\varrho_{K_1}+\varrho_{K_2}\in(\zeta_{0,K_1}-\delta,\zeta_{0,K_1}+\delta).
$$
We then deduce from \eqref{p063c} that
\begin{equation*}
\begin{split}
|U_0(K_2)-U_0(K_1)|&=|U_{0,K_2}(\zeta_{0,K_2})-U_{0,K_1}(\zeta_{0,K_1})|\\
&=|\bar{U}_{0,K_2}(\zeta_{0,K_2}-\varrho_{K_1}+\varrho_{K_2})-U_{0,K_1}(\zeta_{0,K_1})|\\
&\leq|(\bar{U}_{0,K_2}-U_{0,K_1})(\zeta_{0,K_2}-\varrho_{K_1}+\varrho_{K_2})|\\
&+|U_{0,K_1}(\zeta_{0,K_2}-\varrho_{K_1}+\varrho_{K_2})-U_{0,K_1}(\zeta_{0,K_1})|\\
&\leq\sup\limits_{\zeta\in[\zeta_{0,K_1}-\delta,\varrho_{K_1}]}|(\bar{U}_{0,K_2}-U_{0,K_1})(\zeta)|\\
&+\sup\limits_{\zeta\in[\zeta_{0,K_1}-\delta,\zeta_{0,K_1}+\delta]}|U_{0,K_1}(\zeta)-U_{0,K_1}(\zeta_{0,K_1})| \\
&\leq \delta + \sup\limits_{\zeta\in[\zeta_{0,K_1}-\delta,\zeta_{0,K_1}+\delta]} |U_{0,K_1}(\zeta)-U_{0,K_1}(\zeta_{0,K_1})|.
\end{split}
\end{equation*}
Owing to the continuity of $U_{0,K_1}$ at $\zeta_{0,K_1}$, we may pass to the limit $\delta\to 0$ in the above inequality and thereby obtain the first statement of Proposition~\ref{prop.cd2}.

\medskip

\noindent (b) A completely analogous construction is performed in order to prove the continuity of $U_1$ at any $K>0$ satisfying $U_1(K)>1$.
\end{proof}

The just established continuity properties of $U_0$ allows us to derive additional information on $K_0$ and $K_\infty$.

%%%%%%%%%%%%%%%%%%%%
\begin{corollary}\label{cor.cd}
Let $K_0$ and $K_{\infty}$ be defined in Lemma~\ref{lem.U0}. Then $U_0(K_0)=1$ and, either $K_0=K_\infty$, or $K_0<K_\infty$ with $U_0(K_{\infty})=\infty$.
\end{corollary}
%%%%%%%%%%%%%%%%%%%%

\begin{proof}
	Assume first for contradiction that $U_0(K_0)\in (1,\infty)$. Then Propositions~\ref{prop.monot}~(b) and~\ref{prop.cd}~(a) entail that there is $\epsilon_0\in(0,K_0/2)$ such that $1<U_0(K_0-\epsilon)<U_0(K_0)<\infty$ for $\epsilon\in (0,\epsilon_0)$, which contradicts Lemma~\ref{lem.U0}. Assume next for contradiction that  $U_0(K_0)=\infty$. Then, for $K\in (0,K_0)$, we argue as in the proof of~\eqref{p061} in Proposition~\ref{prop.cd}~(a) to prove that
	\begin{equation*}
		\sup_{u\in [0,1]} |(b_{0,K}-b_{0,K_0})(u)| \le (m-1) |C_1(K)-C_1(K_0)|.
	\end{equation*}
	In particular,
	\begin{equation*}
		b_{0,K_0}(1) = \lim_{K\to K_0} b_{0,K}(1) = \lim_{K\to K_0} b_{0,K}(U_0(K)) = 0,
	\end{equation*}
	so that $V_{0,K_0}(U_{0,K_0}^{-1}(1))=0$, which contradicts Lemma~\ref{lem.U0K}~(c). Therefore, $U_0(K_0)=1$.
	
	We now turn to $K_\infty$ and assume for contradiction that $U_0(K_\infty)\in (1,\infty)$. It then follows from Proposition~\ref{prop.cd}~(a) that there is $\epsilon_\infty>0$ such that $U_0(K)<2 U_0(K_\infty)<\infty$ for $K\in (K_\infty,K_\infty+\epsilon_\infty)$, and a contradiction.
\end{proof}

It remains to discard the possibility $K_0=K_\infty$ and prove the continuity of $K\mapsto U_0(K)$ at $K=K_0$ and $K= K_{\infty}$.

%%%%%%%%%%%%%%%%%%%%
\begin{proposition}\label{prop.cd2}
We have
\begin{equation*}
	\lim\limits_{K\to K_0}U_0(K)=1, \quad \lim\limits_{K\to K_{\infty}}U_0(K)=\infty.
\end{equation*}
\end{proposition}
%%%%%%%%%%%%%%%%%%%%

\begin{proof}
We split the proof into several steps, first showing that $K_0<K_\infty$ before moving to the right continuity of $U_0$ at $K_0$ and the left continuity of $U_0$ at $K_\infty$. We begin with the study of some particular orbits of $\Psi_K$ for $K>0$.

\medskip

\noindent \textbf{Step 1.} For $K>0$ and $U_0\in (1,\infty)$, we introduce, for simplicity, the notation
\begin{equation*}	
	(U_{K},V_{K}) := \Psi_{K}(\cdot,U_0,0),
\end{equation*}	
for the orbit of the system~\eqref{p030} passing through the point $(U_0,0)$, defined on the maximal interval $(\zeta_{K}^{-},\zeta_{K}^+)$ with $\zeta_{K}^{-} := \zeta_{K}^{-}(U_0,0)<0$ and $0<\zeta_{K}^{+} := \zeta_{K}^{+}(U_0,0)$. Since
\begin{equation*}	
	\frac{dV_K}{d\zeta}(0) =U_0 - U_0^{2\nu}<0,
\end{equation*}	
we conclude that $V_K$ is negative in a right neighborhood of $\zeta=0$ and we set
\begin{equation*}
	\zeta_K := \inf\big\{ \zeta\in (0,\zeta_K^+)\ :\ V_K(\zeta)=0 \big\} > 0.
\end{equation*}
We particularize now and first focus on the behavior of $(U_{K_0},V_{K_0})$. Assume for contradiction that $\zeta_{K_0}<\zeta_{K_0}^+$. Then, as $(U_{K_0},V_{K_0})$ is not identically equal to either $Q_0$ or $Q_2$, we conclude that
\begin{equation*}
	V_{K_0}(\zeta_{K_0}) = 0, \quad U_{K_0}(\zeta_{K_0}) < 1, \quad \frac{dV_{0,K_0}}{d\zeta}(\zeta_{K_0}) = U_{K_0}(\zeta_{K_0}) - U_{K_0}^{2\nu}(\zeta_{K_0}) > 0.
\end{equation*}
Recalling that $U_{K_0}(0)=U_0>1$, there exists thus $\bar{\zeta}\in (0,\zeta_{K_0})$ such that $U_{K_0}(\bar{\zeta})=1$. Since $V_{K_0}(\bar{\zeta})<0$, the above properties ensure that there is $\delta>0$ such that
\begin{align*}
	& V_{K_0}(\zeta) > b_{0,K_0}(U_{K_0}(\zeta)), \qquad \zeta\in \big( \zeta_{K_0}-\delta, \zeta_{K_0} \big), \\
	& V_{K_0}(\zeta) < b_{0,K_0}(U_{K_0}(\zeta)), \qquad \zeta\in \big( \bar{\zeta}, \bar{\zeta}+\delta \big).
\end{align*}
The continuity of $V_{0,K_0}-b_{0,K_0}(U_{K_0})$ then guarantees that there is $\tilde{\zeta}\in \big( \bar{\zeta}+\delta , \zeta_{K_0}-\delta \big)$ such that $V_{0,K_0}(\tilde{\zeta}) = b_{0,K_0}(U_{K_0}(\tilde{\zeta}))$, which contradicts the separation property~\eqref{y005}  of $\mathcal{C}_{0,K_0}$. Therefore, $\zeta_{K_0} = \zeta_{K_0}^+$ and we have shown that
\begin{equation}
	V_{K_0}(\zeta) < 0, \quad \zeta\in \big( 0 , \zeta_{K_0}^+ \big). \label{lt001}
\end{equation}
Using~\eqref{p030}, we readily infer from~\eqref{lt001} that $U_{K_0}$ is decreasing on $\big( 0 , \zeta_{K_0}^+ \big)$, whence
\begin{equation}
	0< U_{K_0}(\zeta) < U_0, \quad \zeta\in \big( 0 , \zeta_{K_0}^+ \big). \label{lt002}
\end{equation}

We now fix $\overline{V}_0>0$ such that
\begin{equation}
	[C_1(K_0+1) - C_2 U_0^\nu] v + U_0 - \frac{v^2}{2(m-1)} < 0, \qquad v\in \big(-\infty,-\overline{V}_0 \big), \label{lt003}
\end{equation}
and assume for contradiction that $V_{K_0}(\zeta)>-\overline{V}_0$ for any $\zeta\in(0,\zeta_{K_0}^+)$. In this case, both $U_{K_0}$ and $V_{K_0}$ are bounded on $\big(0,\zeta_{K_0}^+\big)$, from which we deduce that $\zeta_{K_0}^+=\infty$. The monotonicity of $U_{K_0}$ ensures additionally that there is $U_\infty\in [0,U_0)$ such that
\begin{equation*}
	\lim_{\zeta\to\infty} U_{K_0}(\zeta)=U_\infty.
\end{equation*}
We then argue as at the end of the proof of \textbf{Step~2} of Proposition~\ref{prop.Ksmall} to conclude that the $\omega$-limit set $\omega_{K_0}(U_0,0)$ of $(U_{K_0},V_{K_0})$ is a critical point of~\eqref{p030} which lies in $[0,U_0)\times (-\infty,0]$; that is, $\omega_{K_0}(U_0,0) = \{Q_0\}$ or $\omega_{K_0}(U_0,0) = \{Q_2\}$. Since $K_0\in (0,K_u]$ by Corollary~\ref{cor.cd}, it follows from Lemma~\ref{lem.Q2} that $Q_2$ is an unstable node and thus the orbit $\Psi_{K_0}(\cdot,U_0,0)$ cannot connect to it. Consequently, $\omega_{K_0}(U_0,0) = \{Q_0\}$ and the uniqueness granted by Lemma~\ref{lem.P0} entails that the orbit $\Psi_{K_0}(\cdot,U_0,0)$ coincides with $l_0(K_0)$. In particular, $(U_0,0)\in l_0(K)$, which contradicts Lemma~\ref{lem.amplitude}. Consequently, there is $\bar{\zeta}_{K_0}\in \big(0,\zeta_{K_0}^+\big)$ such that $V_{K_0}\big(\bar{\zeta}_{K_0}\big)<-\overline{V}_0$. Now, as long as $V_{K_0}(\zeta)<-\overline{V}_{0}$, we infer from~\eqref{p030}, \eqref{lt002}, \eqref{lt003}, and the monotonicity of $K\mapsto C_1(K)$ that
\begin{align*}
	\frac{dV_{K_0}}{d\zeta}(\zeta) & \le C_1(K_0) V_{K_0}(\zeta) - \frac{V_{K_0}^2(\zeta)}{m-1} + U_0 - C_2 U_0^\nu V_{K_0}(\zeta) \\
	& \le C_1(K_0+1) V_{K_0}(\zeta) - \frac{V_{K_0}^2(\zeta)}{m-1} + U_0 - C_2 U_0^\nu V_{K_0}(\zeta) \le - \frac{V_{K_0}^2(\zeta)}{2(m-1)},
\end{align*}
from which we deduce that
\begin{equation}
	V_{K_0}(\zeta) < - \overline{V}_0, \quad \frac{dV_{K_0}}{d\zeta}(\zeta) \le - \frac{V_{K_0}^2(\zeta)}{2(m-1)}, \qquad \zeta\in \big( \bar{\zeta}_{K_0} , \zeta_{K_0}^+\big). \label{lt004}
\end{equation}
Integrating the previous differential inequality leads us to the upper bound
\begin{equation*}
	V_{K_0}(\zeta) \le \frac{2(m-1) V_{K_0}\big( \bar{\zeta}_{K_0} \big)}{ 2(m-1) + V_{K_0}\big( \bar{\zeta}_{K_0} \big) \left( \zeta - \bar{\zeta}_{K_0} \right)} < 0
\end{equation*}
for
\begin{equation*}
	\zeta\in \left( \bar{\zeta}_{K_0} , \bar{\zeta}_{K_0} - \frac{2(m-1)}{V_{K_0}\big( \bar{\zeta}_{K_0} \big)} \right) \cap \big( \bar{\zeta}_{K_0}, \zeta_{K_0}^+ \big).
\end{equation*}
Therefore, $V_{K_0}$ blows down to $-\infty$ at a finite point; that is,
\begin{equation}
	\zeta_{K_0}^+<\infty, \quad \lim\limits_{\zeta\to \zeta_{K_0}^+} V_{K_0}(\zeta) = -\infty. \label{lt005}
\end{equation}

We next show by a continuous dependence argument that the property~\eqref{lt005} extends to a right neighborhood of $K_0$. To this end, we fix $z_0\in \big(\bar{\zeta}_{K_0} , \zeta_{K_0}^+\big)$ such that
\begin{subequations}\label{lt006}
\begin{equation}
	V_{K_0}(\zeta) \le -\overline{V}_0 - 1, \qquad \zeta\in \big[ z_0 , \zeta_{K_0}^+ \big). \label{lt006a}
\end{equation}
Also, since $\frac{dV_{K_0}}{d\zeta}(0)= U_0 - U_0^{2\nu}<0$, there is $\delta_0\in (0,1)$ such that
\begin{equation}
	\frac{dV_{K_0}}{d\zeta}(\zeta) \le \frac{U_0-U_0^{2\nu}}{2}, \qquad \zeta\in [0,\delta_0]. \label{lt006b}
\end{equation}
Finally, the negativity~\eqref{lt001} of $V_{K_0}$ implies that there is $\epsilon_0\in (0,1)$ such that
\begin{equation}
	V_{K_0}(\zeta) \le - 2 \epsilon_0, \qquad \zeta\in [\delta_0,z_0], \label{lt006c}
\end{equation}
and we may assume without loss of generality that $4\epsilon_0 < U_0^{2\nu}-U_0$. The continuous dependence of~\eqref{p030} with respect to the parameter $K$ provides the existence of $\overline{K}_0\in (K_0,K_0+1)$ such that, for each $K\in \big(K_0,\overline{K}_0\big)$, one has $\zeta_K^+>z_0$ and
\begin{align}
	\left| (U_K-U_{K_0})(\zeta) \right| + \left| (V_K-V_{K_0})(\zeta) \right| & \le \epsilon_0, \qquad \zeta\in [0,z_0], \label{lt006d} \\
	\left| \left( \frac{dV_K}{d\zeta} - \frac{dV_{K_0}}{d\zeta} \right)(\zeta) \right| & \le \epsilon_0, \qquad \zeta\in [0,z_0]. \label{lt006e}
\end{align}
\end{subequations}
Now, for $K\in \big(K_0,\overline{K}_0\big)$, we infer from~\eqref{lt006} that
\begin{align*}
	V_K(\zeta) = \int_0^\zeta \frac{dV_K}{d\zeta}(\zeta_*)\, d\zeta_* & \le \int_0^\zeta \left( \frac{U_0-U_0^{2\nu}}{2} + \epsilon_0 \right)\, d\zeta_* \\
	& \le \frac{U_0-U_0^{2\nu}}{4}\zeta < 0, \qquad \zeta\in (0,\delta_0],
\end{align*}
\begin{equation*}
	V_K(\zeta) \le -2\epsilon_0+\epsilon_0 \le -\epsilon_0, \qquad \zeta\in [\delta_0,z_0],
\end{equation*}
and
\begin{equation}
	V_K(z_0) \le - \overline{V}_0  - 1 + \epsilon_0 \le - \overline{V}_0. \label{lt007}
\end{equation}
In particular,
\begin{equation}
	U_K(\zeta)\le U_0, \quad V_K(\zeta)\le 0, \qquad \zeta\in (0,z_0]. \label{lt008}
\end{equation}
Owing to~\eqref{lt007} and \eqref{lt008}, we may argue as in the proof of~\eqref{lt004} with the help of~\eqref{lt003} and the bound $C_1(K)>C_1(K_0+1)$ to conclude that
\begin{equation}
	U_K(\zeta)<U_0, \quad V_K(\zeta) < - \overline{V}_0, \quad \frac{dV_K}{d\zeta}(\zeta) \le - \frac{V_K^2(\zeta)}{2(m-1)}, \qquad \zeta\in \big( z_0 , \zeta_K^+\big), \label{lt009}
\end{equation}
using in addition that the negativity of $V_K$ guarantees the upper bound on $U_K<U_0$. Thanks to~\eqref{lt009}, we proceed as in the proof of~\eqref{lt005} to deduce that, for $K\in \big(K_0,\overline{K}_0\big)$
\begin{equation}
	\zeta_K^+<\infty, \quad \lim\limits_{\zeta\to \zeta_K^+} V_{K}(\zeta) = -\infty. \label{lt010}
\end{equation}

\medskip

\noindent \textbf{Step 2.} Let us now prove that $K_0<K_\infty$ and assume for contradiction that $K_0=K_\infty$, so that $U_0(K_0)=U_0(K_{\infty})=1$ by Corollary~\ref{cor.cd}. Pick $K\in \big(K_0,\overline{K}_0\big)$ and observe that, since $K>K_0=K_\infty$,
\begin{equation*}
	\mathcal{C}_{0,K} = \left\{ v = b_{0,K}(u)\ :\ u\in [0,\infty) \right\} \subset [0,\infty)\times (-\infty,0)
\end{equation*}
is a separatrix for $\Psi_K$ due to~\eqref{y005}, $U_0(K)=\infty$, and Lemma~\ref{lem.U0K}~(c). Therefore, $V_K(\zeta)>b_{0,K}(\zeta)$ in a right neighborhood of $\zeta=0$ and the separation property of $\mathcal{C}_{0,K}$ entails that
\begin{equation}
	V_K(\zeta) > b_{0,K}(\zeta), \qquad \zeta\in \big( 0 , \zeta_K^+ \big). \label{lt011}
\end{equation}
Combining~\eqref{lt008}, \eqref{lt009}, and~\eqref{lt011} entails that
\begin{equation*}
	V_K(\zeta) \ge \min_{[0,U_0]} b_{0,K} > -\infty, \qquad \zeta\in \big( 0 , \zeta_K^+ \big),
\end{equation*}
and contradicts~\eqref{lt010}.

\medskip

We have thus proved that $K_0<K_{\infty}$ and an application of Corollary~\ref{cor.cd} imply that $U_0(K_0)=1$ and $U_0(K_{\infty})=\infty$. We can now move on to the proof of the continuity of $U_0$ as $K\to K_0$ and $K\to K_{\infty}$.

\medskip

\noindent \textbf{Step 3. Continuity as $K\searrow K_0$.} Assume for contradiction that
$$
\overline{U}_0:=\lim\limits_{K\searrow K_0}U_0(K)=\inf\limits_{K>K_0}U_0(K)>1=U_0(K_0).
$$
Pick $U_0\in(1,\overline{U}_0)$ and $K\in \big(K_0,\overline{K}_0\big)$. Then $U_0(K)>\overline{U}_0>U_0>1$ and we infer from Lemma~\ref{lem.U0K}~(a) that $(U_K,V_K)(\zeta)\in \mathcal{B}_{0,K}^+$  in a right neighborhood of $\zeta=0$. Owing to the separation property~\eqref{y005} and the negativity~\eqref{lt008}-\eqref{lt009} of $V_K$, we conclude as above that
\begin{equation*}
	0 > V_K(\zeta) > b_{0,K}(U_K(\zeta)) \ge \min_{[0,U_0]} b_{0,K} > -\infty, \qquad \zeta\in \big( 0 , \zeta_K^+ \big),
\end{equation*}	
which contradicts~\eqref{lt010}. Therefore,
\begin{equation*}
	\lim\limits_{K\searrow K_0} U_0(K) =1,
\end{equation*}
from which the continuity of $U_0$ at $K_0$ follows, recalling that $U_0(K)=1$ for $K\in (0,K_0]$ by Lemma~\ref{lem.U0} and Corollary~\ref{cor.cd}.

\medskip

\noindent \textbf{Step 4.} We are left with proving the continuity as $K\nearrow K_{\infty}$. Assume for contradiction that
\begin{equation*}
	U_0^{\infty}:=\lim\limits_{K\nearrow K_\infty}U_0(K)=\sup\limits_{K<K_{\infty}}U_0(K)<\infty
\end{equation*}
and observe that $U_0^\infty>1$ due to $K_\infty>K_0$. Pick $\epsilon\in (0,(U_0^\infty-1)/2)$. Then there is $\overline{K}_\infty<K_\infty$ such that
\begin{equation*}
	1 < U_0^\infty - 2\epsilon < U_0^\infty - \epsilon < U_0(K) < U_0^\infty, \qquad K\in \big( \overline{K}_\infty,K_\infty \big).
\end{equation*}
Recalling that $b_{0,K}$ solves the ordinary differential equation
\begin{equation}
	b_{0,K}'(u)=\frac{\nu C_1(K)}{u}-\frac{\nu b_{0,K}(u)}{(m-1)u}+\frac{\nu(1-u^{2\nu-1})}{b_{0,K}(u)}-\nu C_2 u^{\nu-1}, \quad u\in \big(0,U_0(K)\big), \label{lt012}
\end{equation}
see~\eqref{p056}, and satisfies
\begin{equation*}
	\big| b_{0,K}(1) - b_{0,K_\infty}(1) \big| \le (m-1) \left( C_1(K) - C_1(K_\infty) \right), %\label{lt013}
\end{equation*}
a continuous dependence argument then entails that
\begin{equation}
	\lim\limits_{K\nearrow K_\infty} \sup_{u\in [1,U_0^\infty-\epsilon]} \left| (b_{0,K} - b_{0,K_\infty})(u) \right| = 0. \label{lt014}
\end{equation}
Consider next $K\in \big( \overline{K}_\infty,K_\infty \big)$. It follows from~\eqref{lt012}, the negativity of $b_{0,K}$, and Young's inequality that, for $u\in [1,U_0(K))$,
\begin{align}
	\big( b_{0,K}^2 \big)'(u) & = 2\nu \left( \frac{C_1(K)}{u} - C_2 u^{\nu-1} \right) b_{0,K}(u) - \frac{2\nu}{m-1} \frac{b_{0,K}^2(u)}{u} - 2\nu \left( u^{2\nu-1}-1 \right) \label{lt015}\\
	& \le \frac{\nu}{m-1} \frac{b_{0,K}^2(u)}{u} + \nu(m-1) C_2^2 u^{2\nu-1} - \frac{2\nu}{m-1} \frac{b_{0,K}^2(u)}{u} \nonumber\\
	& \le \frac{\nu}{(m-1)u} \left[ (m-1)^2 C_2^2 U_0(K)^{2\nu} - b_{0,K}^2(u) \right] \nonumber\\
	& \le \frac{\nu}{(m-1)u} \left[ (m-1)^2 C_2^2 \big(U_0^\infty\big)^{2\nu} - b_{0,K}^2(u) \right]. \nonumber
\end{align}
The comparison principle then entails that
\begin{equation}
	b_{0,K}(u) \le \max{\left\{ |b_{0,K}(1)| , (m-1) C_2 \big( U_0^\infty \big)^\nu \right\}} \le \mathcal{M}_0, \quad u\in [1,U_0(K)], \label{lt016}
\end{equation}
with $\mathcal{M}_0 := \max{\left\{ |b_{0,K_\infty}(1)| + (m-1) C_1\big(\overline{K}_\infty\big) , (m-1) C_2 \big( U_0^\infty \big)^\nu \right\}}>0$.
We next infer from~\eqref{lt015}, \eqref{lt016}, and the negativity of $b_{0,K}$ that, for $u\in [1,U_0(K))$,
\begin{equation*}
	\big( b_{0,K}^2 \big)'(u) \ge\frac{2\nu C_1(K)}{u} b_{0,K}(u) - \frac{2\nu}{m-1} \frac{b_{0,K}^2(u)}{u} - 2\nu \left( u^{2\nu-1}-1 \right) \ge - \mathcal{N}_0
\end{equation*}
with
\begin{equation*}
	\mathcal{N}_0 := 2\nu C_1(K_\infty) \mathcal{M}_0 + \frac{2\nu}{m-1} \mathcal{M}_0^2 + 2\nu \left( \big(U_0^\infty\big)^{2\nu-1} - 1 \right) >0.
\end{equation*}
Integrating the above differential inequality over $(u,U_0(K))$ leads us to the estimate
\begin{equation*}
	b_{0,K}^2(u) \le \mathcal{N}_0 \big( U_0(K) - u \big), \qquad u\in [1,U_0(K)].
\end{equation*}
In particular,
\begin{equation*}
	0 \le b_{0,K}^2\big(U_0^\infty - 2\epsilon) \le \mathcal{N}_0 \big( U_0(K) - U_0^\infty + 2 \epsilon \big)
\end{equation*}
for $K\in \big( \overline{K}_\infty,K_\infty \big)$ and~\eqref{lt014} allows us to let $K\to K_\infty$ in the above inequality to conclude that
\begin{equation*}
	0 \le b_{0,K_\infty}^2\big(U_0^\infty - 2\epsilon) \le 2 \epsilon\mathcal{N}_0.
\end{equation*}
As $\epsilon\in (0,(U_0^\infty-1)/2)$ is arbitrary and $b_{0,K_\infty}\in C([0,\infty))$, we finally let $\epsilon\to 0$ and end up with $b_{0,K_\infty}\big(U_0^\infty)=0$, which contradicts Lemma~\ref{lem.U0K}~(c). Consequently,
\begin{equation*}
	\lim\limits_{K\nearrow K_\infty} U_0(K) = \infty,
\end{equation*}
and the proof of Proposition~\ref{prop.cd2} is complete.
\end{proof}

%%%%%%%%%%%%%%%%%%%%
%%%%%%%%%%%%%%%%%%%%
\section{Proofs of Theorems~\ref{th.1} and~\ref{th.2}}\label{sec.proofs}
%%%%%%%%%%%%%%%%%%%%
%%%%%%%%%%%%%%%%%%%%

We are now in a position to complete the proofs of the two theorems stated in the Introduction, which is the aim of this short section.

\begin{proof}[Proof of Theorem~\ref{th.1}]
Consider the function $U:=U_0-U_1$ on the interval $[K_0,K_{\infty})$. Since $K_0\in (0,K_u]$ by Corollary~\ref{cor00}~(b), the critical point $Q_2$ is unstable and thus $U_1(K_0)>1$. Recalling that $U_0(K_0)=1$ by Lemma~\ref{lem.U0}, we conclude that $U(K_0)=1-U_1(K_0)<0$. Moreover, $\lim\limits_{K\to K_{\infty}}U(K)=\infty$ by Corollary~\ref{cor1} and Proposition~\ref{prop.cd2}. Since $U$ is a continuous function on $[K_0,K_{\infty})$ by Propositions~\ref{prop.cd} and~\ref{prop.cd2}, an application of Bolzano's Theorem guarantees the existence and uniqueness of $K_*\in(K_0,K_{\infty})$ such that $U(K_*)=0$; that is, $l_1(K_*)\cup l_0(K_*)$ is a complete orbit of the system~\eqref{p030} with $K=K_*$, connecting the critical points $Q_1$ and $Q_0$ and $K_*$ is the only value of the parameter $K$ for which such a complete trajectory exists. We now set
\begin{equation*}
	(Y_*,Z_*) := (Y_{0,K_*},Z_{0,K_*}) = (Y_{1,K_*},Z_{1,K_*}), \qquad X_* := Z_*^\nu,
\end{equation*}
and recall that Lemma~\ref{lem.P1} and Lemma~\ref{lem.P0} ensure that
\begin{subequations}\label{y011}
\begin{equation}
	\lim\limits_{\eta\to-\infty} (Y_*,Z_*)(\eta) = (m-1,0), \qquad X_*\in L^1(-\infty,0), \label{y011a}
\end{equation}
and
\begin{equation}
	\lim\limits_{\eta\to \infty} (Y_*,Z_*)(\eta) = (0,0), \qquad X_*\in L^1(0,\infty), \label{y011b}
\end{equation}
\end{subequations}
respectively. We then invert the implicit definition of the independent variable~\eqref{PSvar.large} and the change of variable~\eqref{change.large} to define $\xi$ in terms of $\eta\in\mathbb{R}$ and then $f_*$ by
\begin{equation*}
	\xi(\eta)=\exp\left[\int_{0}^{\eta}X_*(s)\,ds\right], \quad f_*(\xi(\eta)) := \left[\frac{\alpha X_*(\eta)}{2m}\right]^{1/(m-1)}\xi(\eta)^{2/(m-1)},
\end{equation*}
and
\begin{equation*}
 \big(f_*^{m-1}\big)'(\xi(\eta)) := \frac{\alpha(m-1)}{2m} \xi(\eta) Y_*(\eta).
\end{equation*}
It follows from~\eqref{y011} that
\begin{equation*}
	\xi_* := \lim\limits_{\eta\to -\infty} \xi(\eta) = \exp{\left[ - \int_{-\infty}^0 X_*(s)\,ds \right]}>0
\end{equation*}
and
\begin{equation*}
	\xi_0 := \lim\limits_{\eta\to \infty} \xi(\eta) = \exp{\left[ \int_0^{\infty} X_*(s)\,ds \right]}<\infty.
\end{equation*}
Hence, using again~\eqref{y011}, we obtain
\begin{equation*}
	f_*(\xi_*) = f_*(\xi_0) = \big(f_*^{m-1}\big)'(\xi_0) = 0, \quad  \big(f_*^{m-1}\big)'(\xi_*) = \frac{\alpha(m-1)^2}{2m} \xi_*,
\end{equation*}
from which we deduce that $f_*$ satisfies the interface conditions $\big(f_*^{m}\big)'(\xi_*) = \big(f_*^{m}\big)'(\xi_0) = 0$. Finally, it follows from the properties of $Z_*$ and the definition of $f_*$ that $f_*>0$ on $(\xi_*,\xi_0)$ and we have thus completed the proof of Theorem~\ref{th.1}.
\end{proof}

As already pointed out, Theorem~\ref{th.2} follows from Corollary~\ref{cor00} and Lemma~\ref{lem.U0}.

\begin{proof}[Proof of Theorem \ref{th.2}]
Pick $K\in(0,K_0]$. Corollary~\ref{cor00}~(b) and Lemma~\ref{lem.U0} entail that the trajectory $l_0(K)$ is a complete orbit connecting $Q_2$ to $Q_0$. It then follows from~\eqref{p029} that $Z_{0,K}(\eta)\to Z_K$ as $\eta\to -\infty$. As in the proof of Theorem~\ref{th.1}, we invert the implicit definition of the independent variable~\eqref{PSvar.large} and the change of variable~\eqref{change.large} to define $\xi$ in terms of $\eta\in\mathbb{R}$ and then $f$ by
\begin{equation*}
	\xi(\eta)=\exp\left[\int_{0}^{\eta}Z_{0,K}^\nu(s)\,ds\right], \quad f(\xi(\eta)) := \left[\frac{\alpha Z_{0,K}^{\nu}(\eta)}{2m}\right]^{1/(m-1)}\xi(\eta)^{2/(m-1)},
\end{equation*}
and
\begin{equation*}
	\big(f^{m-1}\big)'(\xi(\eta)) := \frac{\alpha(m-1)}{2m} \xi(\eta) Y_{0,K}(\eta).
\end{equation*}
Since $Z_{0,K}^\nu(\eta)\to Z_K^{\nu}>0$ as $\eta\to -\infty$ by Corollary~\ref{cor00} and $Z_{0,K}^\nu\in L^1(0,\infty)$ by Lemma~\ref{lem.P0}, we have
\begin{equation*}
	\lim\limits_{\eta\to -\infty} \xi(\eta) = 0, \qquad \xi_0 :=  \lim\limits_{\xi\to\infty} \xi(\eta) = \exp{\left[ \int_0^\infty Z_{0,K}^\nu(s)\,ds \right]} < \infty,
\end{equation*}
and the local behavior~\eqref{beh.P2} as $\xi\to0$ is then deduced from the definition of $f$. Next, since $Z_{0,K}(\eta)\to 0$ as $\eta\to\infty$, we find that $f(\xi_0)=0$, from which we conclude that the profile $f$ defined above has compact support $[0,\xi_0]$. Finally, we also know from Lemma~\ref{lem.P0} that
\begin{equation*}
	\lim\limits_{\eta\to\infty}\frac{Y_{0,K}(\eta)}{Z_{0,K}(\eta)}=-\frac{K}{m-1}\in(-\infty,0),
\end{equation*}
from which we deduce that
\begin{equation*}
	L :=\lim\limits_{\xi\to\xi_0}\frac{\big(f^{m-1}\big)'(\xi)}{f^{m+q-2}(\xi)}\in (-\infty,0).
\end{equation*}
Since $m+q>2$, we readily conclude that $\big(f^{m}\big)'(\xi_0) = \big(f^{m-1}\big)'(\xi_0) = 0$ and thus $f$ has an interface at $\xi=\xi_0$ according to Definition~\ref{def.dcp}, completing the proof.
\end{proof}

\bigskip

\noindent \textbf{Acknowledgements} This work is partially supported by the Spanish project PID2020-115273GB-I00. Part of this work has been developed during visits of R. G. I. to Institut de Math\'ematiques de Toulouse and Laboratoire de Math\'ematiques LAMA, Universit\'e de Savoie Mont Blanc and of Ph. L. to Universidad Rey Juan Carlos. Both authors thank these institutions for the hospitality and support.

\bibliographystyle{plain}

\end{document}